\documentclass{amsart}
\usepackage{marktext}
\usepackage{fy}

\usepackage[tight]{subfigure}

\colorlet{darkred}{red!40!black}

\begin{document}
\title[nutrient induced tumor boundary instability]%
  {Tumor boundary instability induced by nutrient consumption and supply}
\author[Feng]{Yu Feng}
\address{%
  Yu Feng, Beijing International Center for Mathematical Research, Peking University, No. 5 Yiheyuan Road Haidian District, Beijing, P.R.China 100871}
\email{fengyu@bicmr.pku.edu.cn}
\author[Tang]{Min Tang}
\address{%
  Min Tang:  School of Mathematical Sciences, Institute of Natural Sciences, MOE-LSC, Shanghai Jiao Tong University, Shanghai, 200240, P. R. China}
\email{tangmin@sjtu.edu.cn}
\author[Xu]{Xiaoqian Xu}
\address{Xiaoqian Xu: Zu Chongzhi Center for Mathematics and Computational Sciences, Duke Kunshan University, China}
\email{xiaoqian.xu@dukekunshan.edu.cn}
\author[Zhou]{Zhennan Zhou}
\address{%
  Zhennan Zhou, Beijing International Center for Mathematical Research, Peking University, No. 5 Yiheyuan Road Haidian District, Beijing, P.R.China 100871}
\email{zhennan@bicmr.pku.edu.cn}
\date{\today}

\begin{abstract}
We investigate the tumor boundary instability induced by nutrient consumption and supply based on a Hele-Shaw model derived from taking the incompressible limit of a cell density model. We analyze the boundary stability/instability in two scenarios: 1) the front of the traveling wave; 2) the radially symmetric boundary. In each scenario, we investigate the boundary behaviors under two different nutrient supply regimes, \emph{in vitro} and \emph{in vivo}. Our main conclusion is that for either scenario, the \emph{in vitro} regime always stabilizes the tumor's boundary regardless of the nutrient consumption rate. However, boundary instability may occur when the tumor cells aggressively consume nutrients, and the nutrient supply is governed by the \emph{in vivo} regime. 
\end{abstract}
\maketitle
\section{Introduction}
Tumor, one of the major diseases threatening human life and health, has been widely concerned. The mathematical study of tumors has a long history and constantly active. We refer the reader to the textbook \cites{cristini2010multiscale,cristini2017introduction} and review articles \cites{araujo2004history,lowengrub2009nonlinear,byrne2006modelling,roose2007mathematical}. Previous studies and experiments indicate that the shape of tumors is one of the critical criteria to distinguish malignant from benign. Specifically, malignant tumors are more likely to form dendritic structures than benign ones. Therefore, it is significant to detect and predict the formation of tumor boundary instability through mathematical models. Before discussing the mathematical studies of tumor morphology, we review relevant mathematical models as follows.

The first class of model was initiated by Greenspan in 1976 \cite{greenspan1976growth}, which further inspired a mass of mathematical studies on tumor growth (e.g., \cites{byrne1996growth,chaplain1996avascular,zheng2005nonlinear,friedman2007bifurcationA}). The tumor is regarded as an incompressible fluid satisfying mass conversation. More precisely, these free boundary type models have two main ingredients. One is the nutrient concentration $\sigma$ governed by a reaction-diffusion equation, which considers the consumption by the cells and the supplement by vessels. The other main component is the internal pressure $p$, which further induces the cell velocity $v$ via different physical laws (e.g., Darcy's law     \cites{greenspan1976growth,byrne1996growth,friedman1999analysis,cristini2003nonlinear}, Stokes law \cites{friedman2007bifurcationA,friedman2002quasistatic,friedman2002quasi}, and Darcy\&Stokes law \cites{franks2003interactions,franks2009interactions,zheng2005nonlinear,king2006mathematical,pham2011predictions}). Finally, the two ingredients are coupled via the mass conservation of incompressible tumor cells, which yields the relation $\grad\cdot v=\lambda(\sigma)$, with the cell proliferation rate $\lambda$ depending on $\sigma$. To close the model, the Laplace-Young condition ($p\vert_{\partial\Omega}=\gamma\kappa$, where $\kappa$ is the mean curvature, and $\gamma$ stands for the surface tension coefficient) is imposed on the tumor-host interface. For some  variant models, people replace the Laplace-Young condition with other curvature-dependent boundary conditions (see, e.g., \cites{turian2019morphological,lu2019nonlinear,pham2018nonlinear}). More sophisticated models were also investigated recently. In particular, we mention the studies based on the two-phase models \cites{pham2018nonlinear,turian2019morphological,lu2019nonlinear}, and the works involve chemotaxis \cites{he2022incompressible,kim2022density}. 

Most studies on the stability/instability of tumor boundary are based on the above class of models and have been investigated from different points of view. Among them, for different models (e.g., Darcy \cites{fontelos2003symmetry,friedman2001symmetry,friedman2006bifurcation,friedman2006asymptotic,friedman2008stability}; and Stokes \cite{friedman2006free,friedman2007bifurcationA,friedman2007bifurcationB}), Friedman et al. proved the existence of non-radially symmetric steady states analytically and classified the stability/instability of the boundaries from the Hopf bifurcation point of view. Specifically, in their studies, the bifurcation parameter is characterized by the cell proliferation rate or ratio to cell-cell adhesiveness. Then the authors showed that the boundary stability/instability changes when the parameter crosses a specific bifurcation point. On the other hand, Cristini et al. in \cite{cristini2003nonlinear}, as the pioneers, employ asymptotic analysis to study and predict the tumor evolution. Their work is of great significance to the dynamic simulation of tumors and nurtured more related works in this direction \cites{macklin2007nonlinear,pham2018nonlinear,turian2019morphological,lu2020complex,lu2022nonlinear}. All the research demonstrated that many factors could induce the tumor's boundary instability, including but not limited to vascularization \cites{cristini2003nonlinear,pham2018nonlinear,lu2019nonlinear,lu2022nonlinear}, proliferation \cites{friedman2001symmetry,friedman2008stability,friedman2006free,friedman2007bifurcationB,cristini2003nonlinear,lu2022nonlinear}, apoptosis \cites{friedman2001symmetry,friedman2008stability,friedman2006free,friedman2007bifurcationB,cristini2003nonlinear,turian2019morphological,lu2022nonlinear,pham2018nonlinear,lu2022nonlinear}, cell-cell adhesion \cites{friedman2001symmetry,friedman2008stability,friedman2006free,friedman2007bifurcationB,cristini2003nonlinear,pham2018nonlinear}, bending rigidity \cite{turian2019morphological,lu2019nonlinear}, microenvironment \cite{turian2019morphological,pham2018nonlinear,macklin2007nonlinear,lu2020complex}, chemotaxis \cite{lu2022nonlinear, lu2020complex}.  

In recent decades, tumor modeling from different perspectives has emerged and developed. In particular, one could consider the density model proposed by Byrne and Drasdo in \cite{byrne2009individual}, in which the tumor cell density $\rho$ is governed by a porous medium type equation, and the internal pressure $p$ is induced by the power rule $p=\rho^{m}$ with the parameter $m>1$. The power rule enables $p$ naturally vanish on the tumor boundary. Moreover, the boundary velocity $v$ is governed by Darcy's law $v=-\grad p\vert_{\partial\Omega}$. Previous research indicates that the porous media type equations have an asymptote concerning the parameter m tending to infinity \cites{aronson1998limit,gil2001convergence,igbida2002mesa,kim2003uniqueness,kim2009homogenization}. Motivated by this, Perthame et al. derived the second kind of free boundary model in \cite{perthame2014hele} by taking the incompressible limit (sending $m$ to infinity), or equivalently mesa-limit of the density models. An asymptotic preserving numerical scheme was designed by J.Liu et al. in \cite{liu2018accurate}, the scheme naturally connects the numerical solutions to the density models to that of the free boundary models.

In the mesa-limit free boundary models proposed in \cite{perthame2014hele}, the limit density $\rho_{\infty}$ can only take value in $[0,1]$, and the corresponding limit pressure $p_{\infty}$ is characterized by a monotone Hele-Shaw graph. More specifically, $p_{\infty}$ vanishes on the unsaturated region where $\rho_{\infty}<1$ (see equation \eqref{eqn:HS graph}). The Hele-Shaw graph representation of pressure brings the following advantages. Firstly, in the Hele-Shaw type model, the formation of a necrotic core can be described by an obstacle problem \cite{guillen2022hele}, which leads $\rho_{\infty}$ to decay exponentially in the necrotic core. Due to the Hele-Shaw graph, the pressure $p_{\infty}$ naturally vanishes there instead of taking negative values. Secondly, a transparent regime called "patch solutions" exists, in which $\rho_{\infty}$ remains in the form of $\chi_{D(t)}$, i.e., the indicator function of the tumor region. Again, to satisfy the corresponding Hele-Shaw graph, $p_{\infty}$ has to vanish on the tumor's interface (where $\rho_{\infty}$ drops from $1$ to $0$), which is significantly different from the first kind of free boundary models developed from \cite{greenspan1976growth}, in which the internal pressure relies on the boundary curvature $\kappa$ as mentioned previously. Moreover, in the mesa-limit free boundary models, the boundary velocity is still induced by Darcy's law $v_{\infty}=-\grad p_{\infty}\vert_{\partial\Omega}$. For completeness, the derivation of the mesa-limit model is summarized in Section \ref{sec:model introduction}. Albeit various successful explorations based on such mesa-limit free boundary models \cites{mellet2017hele,david2021free,guillen2022hele,perthame2015incompressible,david2021convergence,kim2003uniqueness,kim2016free,kim2018porous,kim2022density,kim2019singular,david2021incompressible,perthame2014traveling,jacobs2022tumor,liu2021toward,dou2022tumor}, the study on its boundary stability/instability is yet thoroughly open.

The primary purpose of this paper is to investigate whether boundary instability will arise in the mesa-limit free boundary models, which should shed light on the boundary stability of the cell density models when $m$ is sufficiently large. To simplify the discussion, we consider tumors in the avascular stage with saturated cell density so that the density function $\rho_{\infty}$ is a patch solution, and the tumor has a sharp interface. As the first attempt in this regard, we explore the instability caused by nutrient consumption and supply. A similar mechanism can induce boundary instability in other biological systems, see \cite{ben1994generic} for nutrient induce boundary instability in bacterial colony growth models. The role of nutrition in tumor models has been widely studied, and we refer the reader to the latest article in this direction \cite{jacobs2022tumor}. Inspired by \cite{perthame2014traveling}, we divide the nutrient models into two kinds, \emph{in vitro} and \emph{in vivo}, according to the nutrient supply regime. In either regime, the nutrient is consumed linearly in the tumor region with a rate $\lambda>0$. However, in the \emph{in vitro} model, we assume that a liquid surrounds the tumor with nutrient concentration $c_B$. Mathematically, the nutrient concentration remains $c_B$ at the tumor-host interface. For the \emph{in vivo} model, the nutrient is transported by vessels outside the tumor and reaches $c_B$ at the far field. Correspondingly, we assume the exchange rate outside the tumor is determined by the concentration difference from the background, i.e., $c_B-c$. The two nutrient models will be specified more clearly in Section \ref{sec:Two nutrient models}.

Our study of boundary stability/instability consists of two scenarios. We begin with a relatively simple case, the front of traveling waves, in which quantitative properties can be studied more explicitly. In this case, the unperturbed tumor region corresponds to a half plane with the boundary being a vertical line propagating with a constant normal velocity. Then we test the boundary stability/instability by adding a perturbation with frequency $l\in\mathbb{R}^+$ and amplitude $\delta$. Our analysis shows that in the \emph{in vitro} regime, $\delta$ always decreases to zero as time propagates. In other words, the boundary is stable for any frequency perturbation. In contrast, \emph{in vivo} regime, there exists a threshold value $L$ such that the perturbation with a frequency smaller than $L$ becomes unstable when the nutrient consumption rate, $\lambda$, is larger than one. 

The above case corresponds to the boundary stability/instability while the tumor is infinitely large. In order to further explore the influence of the finite size effect on the boundary stability/instability, we consider the perturbation of radially symmetric boundary with different wave numbers $l\in\mathbb{N}$ and radius $R$. Our analysis shows that the \emph{in vitro} regime still suppresses the increase of perturbation amplitude and stabilizes the boundary regardless of the consumption rate, perturbation wave number, and tumor size. For the \emph{in vivo} regime, when the consumption rate $\lambda$ is less than or equal to one, the boundary behaves identically the same as the \emph{in vitro} case. However, when $\lambda$ is greater than one, the continuous growth of tumor radius will enable perturbation wave number to become unstable in turn (from low to high). Further more, as $R$ is approaching infinity, the results in the radial case connect to the counterparts in the traveling wave case.

The main contribution of this work is to show that tumor boundary instability can be induced by nutrient consumption and supply. As a by-product, our results indicate that the cell apoptosis and curvature-dependent boundary conditions present abundantly in previous studies (e.g., \cites{cristini2003nonlinear, friedman2007bifurcationA}) are unnecessary for tumor boundary instability formation.

The paper is organized as follows. In Section \ref{sec:preliminary}, we first derive our free boundary models by taking the incompressible limit of density models characterized by porous medium type equations in Section \ref{sec:model introduction}. Besides that, we also introduce the \emph{in vitro} and \emph{in vivo} nutrient regimes in this subsection. Furthermore, the corresponding analytic solutions are derived in Section \ref{sec:Analytic solutions}. Section \ref{sec:Framework of perturbation analysis} is devoted to introducing the linear perturbation technique in a general framework. Then, by using the technique in Section \ref{sec:Framework of perturbation analysis}, we study the boundary stability of the traveling wave and the radially symmetric boundary under the two nutrient regimes, respectively, in Section \ref{sec:traveling wave} and Section \ref{sec:radial boundary} (with main results in Section \ref{sec:Set up for the traveling wave perturbation} and Section \ref{sec:Set up for the radial boundary perturbation}). Finally, we summarize our results and discuss future research plans in Section \ref{Conclusion}.

\section{Preliminary}
\label{sec:preliminary}
\subsection{model introduction}
\label{sec:model introduction}
\subsubsection{The cell density model and its Hele-Shaw limit}
To study the tumor growth under nutrient supply, let $\rho(x,t)$ denote the cell population density and $c(x,t)$ be the nutrient concentration. We assume the production rate of tumor cells is given by the growth function $G(c)$, which only depends on the nutrient concentration. On the other hand, we introduce
\begin{equation}
\label{eqn:set D}
    D(t)=\left\{\rho(x,t)>0\right\}
\end{equation}
to denote the support of $\rho$. Physically, it presents the tumoral region at time $t$. We assume the tumoral region expands with a finite speed governed by the Darcy law $v=-\grad p$ via the pressure $p(\rho)=\rho^m$. Thus, the cell density $\rho$ satisfies the equation:
\begin{equation}
\label{eqn:pme density eqn}
    \frac{\partial}{\partial t}\rho-\grad\cdot\left(\rho\grad p(\rho)\right)=\rho G(c),\quad x\in\mathbb{R}^2,\quad t\geq 0.
\end{equation}
For the growth function $G(c)$, we assume 
\begin{equation}
    G(c)=G_0 c,\quad\text{with }\quad G_0>0,
\end{equation}
note that in contrast to the nutrient models in \cites{tang2014composite,perthame2014hele}, we eliminate the possibility of the formation of a necrotic core by assuming that $G(\cdot)$ is always positive and linear (for simplicity), since this project aims to study the boundary instability induced by the nutrient distribution itself.

Many researches, e.g. \cites{perthame2014hele,david2021free,david2021incompressible,kim2016free,kim2018porous,guillen2022hele},  indicate that there is a limit as $m\rightarrow\infty$ which turns out to be a solution to a free boundary problem of Hele-Shaw type. To see what happens, we multiply equation \eqref{eqn:pme density eqn} by $m\rho^{m-1}$ on both sides to get
\begin{equation}
    \frac{\partial}{\partial t}p(\rho)=\abs{\grad p(\rho)}^2+m p(\rho)\lap p(\rho)+m G_0 p(\rho)c.
\end{equation}
Hence, if we send $m\rightarrow\infty$, we formally obtain the so called \emph{complementarity condition} (see \cites{perthame2014hele,david2021free} for a slight different model):
\begin{equation}
\label{eqn:complementary}
    p_{\infty}(\lap p_{\infty}+G_0 c)=0.
\end{equation}
On the other hand, the cell density $\rho(x,t)$ converges to the weak solution (see \cite{perthame2014hele}) of 
\begin{equation}
\label{eqn:rho infty}
       \frac{\partial}{\partial t}\rho_{\infty}-\grad\cdot\left(\rho_{\infty}\grad p_{\infty}\right)=\rho_{\infty} G(c),
\end{equation}
and $p_{\infty}$ compels the limit density $\rho_{\infty}$ only take value in the range of $\left[0,1\right]$ for any initial date $\rho_0\in\left[0,1\right]$ (see Theorem 4.1 in 
\cite{perthame2014hele} for a slightly different model).
Moreover, the limit pressure $p_{\infty}$ belongs to the Hele-Shaw monotone graph:
\begin{equation}
\label{eqn:HS graph}
p_{\infty}(\rho_{\infty})=\left\{
  \begin{array}{rcr}
    0, \qquad 0\leq\rho_{\infty}<1,\\
    \left[0,\infty\right),\qquad \rho_{\infty}=1.\\
  \end{array}
\right.
\end{equation}
The incompressible limit and the \emph{complementarity condition} of a fluid mechanical related model have been rigorously justified in \cites{perthame2014hele,david2021free}. And the incompressible limit of \eqref{eqn:pme density eqn} (coupled with nutrient models that will be introduced in the next section) was verified numerically in \cite{liu2018analysis}. 

We define the support of $p_{\infty}$ to be
\begin{equation}
\label{eqn:set D infty}
    D_{\infty}(t)=\left\{p_{\infty}(x,t)>0\right\},
\end{equation}
then \eqref{eqn:complementary} and \eqref{eqn:HS graph} together yield 
\begin{subequations}
\label{eqn:p_infty}
\begin{alignat}{2}
\label{eqn:p_infty eqn}
-\lap p_{\infty}=G_0 c\qquad&\text{for}\quad x\in D_{\infty}(t),\\
\label{eqn:p_infty bc}
p_{\infty}=0,\qquad&\text{for}\quad x\in \mathbb{R}^2\setminus D_{\infty}(t),
\end{alignat}
and $\rho_{\infty}=1$ in $D_{\infty}$.
\end{subequations}
Therefore, once the nutrient concentration $c(x,t)$ is known one can recover $p_{\infty}$ from the elliptic equation above. 

Now we justify the relationship between $D(t)$ and $D_{\infty}(t)$. Observe that when $m$ is finite, $\rho$ and $p(\rho)$ have the same support $D(t)$, whereas as $m$ tends to infinity, $\rho_{\infty}$ may have larger support than $p_{\infty}$. However, a large class of initial data, see e.g. \cite{perthame2016some}, enable the free boundary problem (correspond to \eqref{eqn:rho infty} and \eqref{eqn:p_infty}) possess patch solutions, i.e., $\rho_{\infty}=\chi_{D_{\infty}}$, where $\chi_{A}$ stands for the indicator function of the set $A$. In this case, the support of $p_{\infty}$ coincides with that of $\rho_{\infty}$. Moreover, the boundary velocity $v$ is governed by Darcy law $v=-\grad p_{\infty}$. Further, the boundary moving speed along the normal direction at the boundary point $x$, denote by $\sigma(x)$, is given by:
\begin{equation}
\label{eqn:general boudary speed}
  \sigma(x)=-\grad p_{\infty}\cdot\hat{n}(x),
\end{equation}
where $\hat{n}(x)$ is the outer unit normal vector at $x\in \partial D_{\infty}(t)$.
The boundary speed for more general initial data was studied in \cite{kim2018porous}. 

As the end of this subsection, we emphasize that in our free boundary model, as the limit of the density models, the pressure $p_{\infty}$ always vanishes on $\partial D_{\infty}$. However, as mentioned previously, in the first kind free boundary models, the internal pressure $\Tilde{p}$ is assumed to satisfy the so-called Laplace-Young condition (or some other curvature dependent boundary condition). Mathematically, the boundary condition \eqref{eqn:p_infty bc} is replaced by
\begin{equation}
\label{eqn:curvature condition}
    \Tilde{p}(x)=\gamma \kappa(x),
\end{equation}
where $\gamma>0$ is a constant coefficient, and $\kappa(x)$ denotes the curvature at the boundary point $x$. In the related studies, the curvature condition \eqref{eqn:curvature condition} plays an essential role (e.g., \cites{cristini2003nonlinear, friedman2007bifurcationA}).
\subsubsection{Two nutrient models}
\label{sec:Two nutrient models}
Regarding the nutrient, it diffuses freely over the two dimensional plane. However, inside the tumoral region, the cells consume the nutrient. While outside the tumor, the nutrient exchanges with the far field concentration $c_B$ provided by the surrounding environment or vasculature. It follows that the following reaction-diffusion equation can govern the consumption, exchange, and diffusion of the nutrient in general:
\begin{equation}
\label{eqn:parabolic nutrient}
    \tau\partial_t c-\lap c +\Psi(\rho,c)=0,
\end{equation}
where $\tau$ is the characteristic time scale of the nutrient change, and $\Psi(\rho,c)$ describes the overall effects of the nutrient supply regime outside the tumor and the nutrient consumption by cells inside the tumor. To simplify the mathematical analysis, we drop the time derivative in \eqref{eqn:parabolic nutrient} and consider following elliptic formulation instead
\begin{equation}
\label{eqn:general nutrient eqn}
    -\lap c +\Psi(\rho,c)=0.
\end{equation}
This is reasonable because $\tau\ll 1$ (see, e.g., \cite{greenspan1972models,adam2012survey,byrne1996growth}). As in \cite{perthame2014traveling}, two specific developed and widely studied models are the \emph{in vitro} and the \emph{in vivo} model.

For the \emph{in vitro} model, we assume that the tumor is surrounded by a liquid in which the exchange rate with the background is so fast that the nutrient concentration can be assumed to be the same constant $c_B$ as that of the surrounding liquid, while inside the tumoral region, the consumption function is bi-linear in both the concentration $c$ and the cell density $\rho$ with consumption rate $\lambda>0$. The boundary instability was observed in models where tissues aggressively consume nutrients \cite{maury2014congestion}. Therefore, in our models, we expect boundaries are more likely to be unstable when $\lambda$ is large.
When the \emph{in vitro} is coupled with the cell density model \eqref{eqn:pme density eqn}, equation \eqref{eqn:general nutrient eqn} writes
\begin{subequations}
\label{eqn:finite m invitro}
\begin{alignat}{2}
-\lap c + \lambda \rho c = 0,\qquad &\text{for}\quad x\in D(t),\\
c =c_B,\qquad &\text{for}\quad x\in \mathbb{R}^2\setminus D(t).
\end{alignat}
\end{subequations}
By considering the incompressible limit of the density model (sending $m\rightarrow\infty$), and concern patch solutions $\rho_{\infty}=\chi_{D_{\infty}}$. Equation \eqref{eqn:finite m invitro} tends to:
\begin{subequations}
\label{eqn:infinite m invitro}
\begin{alignat}{2}
-\lap c + \lambda c = 0,\qquad &\text{for}\quad x\in D_{\infty}(t),\\
c =c_B,\qquad &\text{for}\quad x\in \mathbb{R}^2\setminus D_{\infty}(t).
\end{alignat}
\end{subequations}

For the \emph{in vivo} model, the consumption of nutrients within the tumor region (where $\rho > 0$) remains the same as in the \emph{in vitro} model. However, in the \emph{in vivo} model, the nutrients are provided by vessels of the healthy tissue surrounding the tumor, while the healthy tissue consumes nutrients as well. This leads to the nutrient supply outside the tumor being determined by the concentration difference from the background, $c_B - c$, with a positive coefficient $\Tilde{\lambda}$. Mathematically, the overall function $\Psi(\rho,c)$ is written as $\Psi(\rho,c)=\lambda\rho c\cdot\chi_{D}-\Tilde{\lambda}(c_B-c)\cdot\chi_{D^c}$. For simplicity, we set $\Tilde{\lambda}=1$ and $\lambda>0$. Note that this expression guarantees the nutrient concentration reaches $c_B$ at the far field. A more detailed discussion of this issue can be found in \cites{chatelain2011emergence,jagiella2012parameterization}. 

With the same reason as the previous case, by taking $m\rightarrow\infty$ in the density model and concerning patch solutions, we get the \emph{in vivo} nutrient equations for the limit free boundary model,
\begin{subequations}
\label{eqn:infinite m invivo}
\begin{alignat}{2}
-\lap c + \lambda c = 0,\qquad &\text{for}\quad x\in D_{\infty}(t),\\
-\lap c =c_B-c,\qquad &\text{for}\quad x\in \mathbb{R}^2\setminus D_{\infty}(t).
\end{alignat}
\end{subequations}
Moreover, we need to emphasize that the \emph{in vivo} we refer to is different from the previous articles (see, e.g., \cite{cristini2003nonlinear}) in which \emph{in vivo} corresponds to the vascularization inside the tumor. 

The uneven growth phenomena in the tumor models are conjectured due to the non-uniform distribution of nutrients \cite{maury2014congestion}. More precisely, in contrast to the fingertips region, the nutrient is inadequate around the valley since more cells consume nutrients there. Consequently, the tissue around the tips grows faster than the valleys, and therefore instability occurs. In the \emph{in vitro} model, the concentration of the nutrient will match the background concentration $c_B$ at the boundary regardless of the regions. However, for the \emph{in vivo} model, the nutrient is directly available only from healthy tissue; this regime will enlarge the concentration difference at the tips and valleys. Therefore, we expect tumor borders are more prone to grow unevenly in the \emph{in vivo} models, in particular when the consumption rate $\lambda$ is relatively large.

\subsection{Analytic solutions}
\label{sec:Analytic solutions}
Starting from this section, we focus on the mesa limit free boundary models. Therefore, for simplicity of the notations, we drop the free boundary models' subscripts and use $D(t)$, $\rho$, and $p$ to denote the tumoral region, cell density, and pressure in the limit model. On the other hand, through this paper, we use $I_j$ and $K_j$ ($j\in\mathbb{N}$) to denote the second kind of modified Bessel functions, their definitions and basic properties are reviewed in Appendix \ref{sec:Properties of Bessel functions}.

The models introduced in Section \ref{sec:model introduction} have been studied in \cite{liu2018analysis} when $\lambda=1$. In particular, the authors derived $2D$ radially symmetric solutions for the free boundary models, which are coupled with either the \emph{in vitro} or the \emph{in vivo} model. Moreover, their computation yields that as the radius of the tumor tends to infinity, the boundary velocity tends to be a finite constant. In other words, the radially symmetric solutions converge to traveling wave solutions. 

For self-consistency, we recall the derivation of the radially symmetric solutions in \cite{liu2018analysis} in this section. Besides that we also derive the traveling wave solutions for the two nutrient models and verify that they are indeed the limit of the radially symmetric solutions as radius goes to infinity. The analytical solutions in this section will serve as the cornerstone of subsequent perturbation analysis. Now, we begin with the traveling wave scenario.

\subsubsection{traveling plane solution for the \emph{in vitro} model}
\label{sec:traveling plane solution for the invitro model}
For solving two-dimensional traveling wave solutions, we fix the traveling front at $\xi=x-\sigma t=0$, where $\sigma$ stands for the traveling speed and will be determined later. Without loss of generality, let the tumoral region be the left half plane, that is $D(t)=\left\{(\xi,y)\vert \xi\leq 0\right\}$. One can easily see that in the unperturbed two-dimensional problem, to find its solution reduces to solve a one-dimensional problem. Moreover, we disclose that the variable $y$ will serve as the perturbation parameter in the perturbation problems, which will be investigated later. The one dimensional problem writes:
\begin{subequations}
\label{eqn:TW in vitro}
\begin{alignat}{2}
-\partial^2_{\xi}c + \lambda c = 0,\qquad &\text{for}\quad \xi\leq 0,\\
c =c_B,\qquad&\text{at}\quad \xi= 0,
\end{alignat}
in addition, we also assume the concentration of nutrient vanish at the center of tumor, that is
\begin{equation}
    c(\xi)=0,\qquad \text{for}\quad \xi= -\infty.
\end{equation}
\end{subequations}
And the equations for pressure $p(\xi,y)$, i.e., \eqref{eqn:p_infty} and \eqref{eqn:general boudary speed} reads
\begin{subequations}
\label{eqn:TW p}
\begin{alignat}{2}
-\partial^2_{\xi}p=G_0 c,\qquad&\text{for}\quad \xi\leq 0,\\
p=0,\qquad&\text{for}\quad \xi\geq 0,
\end{alignat}
and traveling speed is given by
\begin{equation}
   \sigma=-\partial_{\xi} p(0). 
\end{equation}
Since the gradient of the pressure is always equal to zero at the center of the tumor, we also have
\begin{equation}
    \partial_{\xi} p(\xi)=0,\qquad \text{for}\quad \xi= -\infty.
\end{equation}
\end{subequations}
By solving \eqref{eqn:TW in vitro} we get
\begin{equation}
\label{eqn:TW nutrient steady soln invitro}
    c=c_B e^{\sqrt{\lambda}\xi},\qquad\text{for}\quad \xi\leq 0,
\end{equation}
plug the above expression into \eqref{eqn:TW p} to solve for $p$ and get:
\begin{equation}
\label{eqn:TW pressure steady soln invitro}
    p(\xi)= -\frac{G_0 c_B}{\lambda}e^{\sqrt{\lambda}\xi}+\frac{G_0 c_B}{\lambda}\qquad\text{for}\quad \xi\leq 0.
\end{equation}
Then, we can further find the traveling speed
\begin{equation}
\label{eqn:TW speed invitro}
    \sigma=-\partial_{\xi}p(0)=\frac{G_0 c_B}{\sqrt{\lambda}}.
\end{equation}

\subsubsection{traveling plane solution for the \emph{in vivo} model}
\label{sec:traveling plane solution for the invivo model}
For the \emph{in vivo} model, the only difference from the \emph{in vitro} model is the equations for $c(\xi,y)$ are replaced by
\begin{subequations}
\label{eqn:TW in vivo}
\begin{alignat}{2}
-\partial^2_{\xi}c + \lambda c = 0,\qquad &\text{for}\quad \xi\leq 0,\\
-\partial^2_{\xi}c = c_B-c,\qquad &\text{for}\quad \xi\geq 0,\\
c(\xi)=0,\qquad&\text{for}\quad \xi= -\infty.
\end{alignat}
in addition, $c$ and $\partial_{\xi} c$ are both continuous at the boundary of the tumor, that is
\begin{equation}
    c(0^-)=c(0^+)\quad\text{and}\quad \partial_{\xi} c(0^-)= \partial_{\xi} c(0^+).
\end{equation}
\end{subequations}
And the pressure $p$ still satisfies \eqref{eqn:TW p}.

By solving \eqref{eqn:TW in vivo} we get
\begin{equation}
\label{eqn:TW nutrient steady soln invivo}
c(\xi)=\left\{
  \begin{array}{rcr}
    \frac{c_B}{\sqrt{\lambda}+1}e^{\sqrt{\lambda}\xi}\defeq c^{\text{(i)}}(\xi)\qquad &\text{for}\quad \xi\leq 0, \\
    -\frac{\sqrt{\lambda}c_B}{\sqrt{\lambda}+1}e^{-\xi}+c_B\defeq c^{\text{(o)}}(\xi)\qquad &\text{for}\quad \xi\geq 0.\\
  \end{array}
\right.
\end{equation}
and plug the above expression into \eqref{eqn:TW p} to derive $p$ and get
\begin{equation}
\label{eqn:TW pressure steady soln invivo}
    p(\xi)= -\frac{G_0 c_B}{\lambda(\sqrt{\lambda}+1)}e^{\sqrt{\lambda}\xi}+\frac{G_0 c_B}{\lambda(\sqrt{\lambda}+1)}\qquad\text{for}\quad \xi\leq 0.
\end{equation}
And the boundary speed is given by
\begin{equation}
\label{eqn:TW speed invivo}
    \sigma=-\partial_{\xi}p(0)=\frac{G_0 c_B}{\lambda+\sqrt{\lambda}}.
\end{equation}
By now, we have finished the derivation for the traveling wave solutions. In the next two subsections, we recall the derivation for the radially symmetric scenario in \cites{liu2018analysis} and verify that the boundary speeds converge to the traveling waves' for the corresponding nutrient regime.
\subsubsection{2D radially symmetric model for the \emph{in vitro} model}
\label{sec:2D radially symmetric model for the in vitro model}
For the radially symmetric free boundary model, the tumoral region becomes $D(t)=\mathbb{B}_{R(t)}(0)$ (a disk centered at origin with radius $R$). In this case, we employ polar coordinates $(r,\theta)$, and we can conclude that the solutions are independent of $\theta$ by symmetry. However the variable $\theta$ will play an important role in the perturbed problem, which will be seen in the later sections. Thus, for the free boundary model with nutrients governed by the \emph{in vitro} model, equation \eqref{eqn:infinite m invitro} can be further written as
\begin{subequations}
\label{eqn:radial in vitro}
\begin{alignat}{2}
-\frac{1}{r}\partial_r(r\partial_r c)  + \lambda c = 0,\qquad &\text{for}\quad r\leq R(t),\\
c =c_B,\qquad &\text{for}\quad r\geq R(t).
\end{alignat}
\end{subequations}
And the equations for pressure $p$ \eqref{eqn:p_infty} and \eqref{eqn:general boudary speed} reads
\begin{subequations}
\label{eqn:radial p}
\begin{alignat}{2}
-\frac{1}{r}\partial_r(r\partial_r p)=G_0 c\qquad&\text{for}\quad r\leq R(t),\\
p=0,\qquad&\text{for}\quad r\geq R(t),\\
\sigma(R(t))=-\partial_r p(R(t)),\qquad&\text{on}\quad\partial \mathbb{B}_{R}(0).
\end{alignat}
And by symmetry, we also require
\begin{equation}
    \partial_r p(0)=0.
\end{equation}
\end{subequations}
By solving \eqref{eqn:radial in vitro} we get
\begin{equation}
\label{eqn:Radial nutrient steady soln invitro}
    c(r,t)=c_B\frac{I_0(\sqrt{\lambda}r)}{I_0(\sqrt{\lambda}R)}\qquad\text{for}\quad r\leq R(t).
\end{equation}
Plug the above expression into \eqref{eqn:radial p} to solve for $p$, and we get:
\begin{equation}
\label{eqn:Radial pressure steady soln invitro}
    p(r,t)=-\frac{G_0 c_B}{\lambda I_0(\sqrt{\lambda}R(t))}I_0(\sqrt{\lambda}r)+\frac{G_0}{\lambda}c_B\qquad\text{for}\quad r\leq R(t).
\end{equation}
And the boundary velocity is given by
\begin{equation}
\label{eqn:radial boundary speed in vitro}
    \dot{R}=\sigma(R(t))=-\frac{\partial p}{\partial r} (R(t))=\frac{G_0 c_B I_1(\sqrt{\lambda}R)}{\sqrt{\lambda}I_0(\sqrt{\lambda}R)}.
\end{equation}
Note that as $R(t)\rightarrow\infty$ the speed limit is $\frac{G_0 c_B}{\sqrt{\lambda}}$, which recovers the speed for the traveling wave solution \eqref{eqn:TW speed invitro}.

\subsubsection{2D radially symmetric model for the \emph{in vivo} model}
\label{sec:2D radially symmetric model for the in vivo model}
The computation is similar to the previous case, except that the equations for nutrient are replaced by
\begin{subequations}
\label{eqn:radial in vivo}
\begin{alignat}{2}
-\frac{1}{r}\partial_r(r\partial_r c)  + \lambda c = 0,\qquad &\text{for}\quad r\leq R(t),\\
-\frac{1}{r}\partial_r(r\partial_r c) = c_B-c,\qquad &\text{for}\quad r\geq R(t).
\end{alignat}
\end{subequations}
By solving above two equations and using the continuity of both $c$ and $\partial_r c$ at $R(t)$, we get
\begin{equation}
\label{eqn:Radial nutrient steady soln invivo}
c(r,t)=\left\{
  \begin{array}{rcr}
    c_B a_0(R) I_0(\sqrt{\lambda}r)\defeq c^{\text{(i)}}(r,t),\qquad &\text{for}\quad r\leq R(t), \\
    c_B(1+b_0(R) K_0(r))\defeq c^{\text{(o)}}(r,t),\qquad &\text{for}\quad r\geq R(t),\\
  \end{array}
\right.
\end{equation}
where $a_0$ and $b_0$ are given by
\begin{subequations}
\label{eqn:Radial a0 and b0}
\begin{alignat}{2}
    a_0(R)&=\frac{K_1(R)}{\sqrt{\lambda}K_0(R)I_1(\sqrt{\lambda}R)+K_1(R)I_0(\sqrt{\lambda}R)}\defeq\frac{K_1(R)}{C(R)},\\
    b_0(R)&=-\frac{\sqrt{\lambda} I_1(R)}{\sqrt{\lambda}K_0(R)I_1(\sqrt{\lambda}R)+K_1(R)I_0(\sqrt{\lambda}R)}\defeq-\frac{\sqrt{\lambda} I_1(R)}{C(R)}.
\end{alignat}
\end{subequations}
Then from the pressure equations \eqref{eqn:radial p}, we can solve and get
\begin{equation}
\label{eqn:Radial pressure steady soln invivo}
    p(r,t)=
    -\frac{G_0 c_B}{\lambda}a_0(R) I_0(\sqrt{\lambda}r)+\frac{G_0 c_B}{\lambda}a_0(R) I_0(\sqrt{\lambda}R),\qquad \text{for}\quad r\leq R(t).
\end{equation}
And the velocity of the boundary is given by
\begin{equation}
\label{eqn:radial boundary speed in vivo}
    \dot{R}=\sigma(R(t))
    =\frac{G_0 c_B K_1(R) I_1(\sqrt{\lambda}R)}{\lambda K_0(R)I_1(\sqrt{\lambda}R)+\sqrt{\lambda}K_1(R)I_0(\sqrt{\lambda}R)}\leq\frac{G_0 c_B I_1(\sqrt{\lambda}R)}{\sqrt{\lambda}I_0(\sqrt{\lambda}R)},
\end{equation}
which implies that the speed in the \emph{in vivo} model is slower than that in the \emph{in vitro} model. Again, by sending $R\rightarrow\infty$, we get the limiting speed for the \emph{in vivo} model is $\frac{c_B G_0}{\lambda+\sqrt{\lambda}}$, which recovers the speed for the traveling wave in \eqref{eqn:TW speed invivo}.

\section{Framework of the perturbation analysis}
\label{sec:Framework of perturbation analysis}
We devote this section to establishing the general framework of our asymptotic analysis. Such analysis involves classical techniques which was originally developed by Mullins et al. in \cite{mullins1963morphological} and widely used in \cites{cristini2003nonlinear,macklin2007nonlinear,pham2018nonlinear,turian2019morphological,lu2020complex,lu2022nonlinear}, whereas we present it as generic methodology which in theory can be applied to other problems as well.

We divide our analysis into three parts as follows. 
\subsection{Perturbation of the boundary}
\label{sec:Perturbation of the boundary}
     We study the perturbation of two kinds of boundaries, the radial boundary and the front of traveling waves, and the relationship between them. In either case, we have a proper coordinate system denoted as $(\zeta,\vartheta)$. For simplicity, we assume the boundary profile is a curve $\mathcal{O}_t\subseteq \mathbb{R}^2$, which can be parameterized by the variable $\vartheta$ in the following form:
     \begin{equation}
         \mathcal{O}_t(\vartheta)=\left\{(\zeta,\vartheta)\vert \zeta=\mathcal{Z}(t,\vartheta), \vartheta\in\mathcal{R}\right\}
     \end{equation}
     with some contour index function $\mathcal{Z}(t,\vartheta)$ and range $\mathcal{R}$.

     For the radial case, the unperturbed tumor region at time $t$ is given by a disk with radius $R(t)$, that is $D(t)=\mathbb{B}_{R(t)}$. In this case, equations and functions are naturally presented in terms of the polar coordinate. Therefore,  $(\zeta,\vartheta)=(r,\theta)$ and $\mathcal{R}=\left[-\pi, \pi\right)$. Further more, the tumor boundary at time $t$ can be written as:
    \begin{equation}
    \label{eqn:radial boundary}
        \mathcal{B}_t(\theta)=\left\{(r,\theta)\vert r=R(t), \theta\in\left[-\pi, \pi\right)\right\}.
    \end{equation}

    For the traveling wave case, we employ the Euler coordinate $(\xi,y)$ (where $\xi=x-\sigma t$). In this case, the tumor region is a half plane with a moving front. We fix the front (propagate to the right) at $\xi=0$ with traveling speed $\sigma$, and the tumor region, therefore, become $D(t)=\left\{(\xi,y)\vert\xi\leq 0\right\}$. Then, we write the traveling front more clearly in the parameter curve form:
    \begin{equation}
    \label{eqn:traveling wave boundary}
        \mathcal{B}(y)=\left\{(\xi,y)\vert\xi=0,  y\in\mathbb{R}\right\}.
    \end{equation} 
    
     For the purpose of introducing perturbation method in a general framework, we combine the two scenarios above in the following unified notations. Let $D(t)$ still presents the tumor region at time $t$; and the boundary curve writes
     \begin{equation}
     \label{eqn:general unperturbed boundary}
        \mathcal{B}_t(\vartheta)=\left\{(\zeta,\vartheta)\vert \zeta=Z(t), \vartheta\in\mathcal{R}\right\}.
     \end{equation}
    Moreover, any point $B\in\mathcal{B}_t$ can be presented as $B(Z,\vartheta_*)$ for some $\vartheta_*\in\mathcal{R}$.
    Note that in either case above, the index function $Z(t)$ is independent on the parameter variable $\vartheta$. More precisely, for the radial case $Z(t)=R(t)$, and \eqref{eqn:general unperturbed boundary} stands for \eqref{eqn:radial boundary}; for the traveling wave case, \eqref{eqn:general unperturbed boundary} stands for \eqref{eqn:traveling wave boundary} with $Z(t)$ takes constant value $0$. 
    
    Next, we add a small perturbation to the two kinds of boundaries. From the parameterization representation point of view, the perturbation replaces the boundary curve \eqref{eqn:general unperturbed boundary} by:
    \begin{equation}
    \label{eqn:boundary perturbation}
        \Tilde{\mathcal{B}}_t(\vartheta)=\left\{(\zeta,\vartheta)\vert\zeta=Z(t)+\delta(t) \mathcal{P}(\vartheta),\vartheta\in\mathcal{R}\right\},
    \end{equation}
    where $\delta(t)\ll 1$ stands for the amplitude of the perturbation, and $\mathcal{P}(\vartheta)$ characterizes the perturbation profile. Thus, the perturbed boundary at time $t$ is still parameterized by the variable $\vartheta$. Intuitively, \eqref{eqn:boundary perturbation} means that the perturbation will push the point $(Z,\vartheta_*)\in\mathcal{B}_t$ to $(Z+\delta\mathcal{P}(\vartheta_*),\vartheta_*)\in\Tilde{\mathcal{B}}_t$ for any $\vartheta_*\in\mathcal{R}$. Note that the perturbation form $\eqref{eqn:boundary perturbation}$ enables the evolution of the perturbation term to reduce to the evolution of the amplitude function $\delta(t)$ while its spatial profile persists. Such an ansatz with temporal and spatial degrees of freedom separated makes sense only when the profile function represents a typical model of a general classical of contours. In the next, we explain how to choose the perturbation profiles in the two cases.
    
    In the radial symmetry case, the profile $\mathcal{P}(\theta)$ is parameterized by $\theta\in[-\pi, \pi)$ and it can be expressed as a Fourier expansions in general. In particular, for the single wave perturbation with wave number $l$, $\mathcal{P}(\theta)$ takes the form of:
    \begin{equation}
        \mathcal{P}(\theta)=C_1\cos{l\theta}+C_2\sin{l\theta},\quad\text{with}\quad l\in\mathbb{N}^+,
    \end{equation}
    where $C_1, C_2$ are constant coefficients. Note that by rotating the coordinate system and rescaling on $\delta(t)$, without loss of generality we can simply take
    \begin{equation}
    \label{eqn:radial perturbation profile}
        \mathcal{P}(\theta)=\cos{l\theta}\defeq\mathcal{P}_{l}(\theta).
    \end{equation}
    
    For the traveling wave case, the profile is parameterized by $y\in\mathbb{R}$. By a similar reason to the radial case, we can simply consider
    \begin{equation}
    \label{eqn:TW perturbation profile}
        \mathcal{P}_{l}(y)=\cos{l y}, \quad\text{with}\quad l\in\mathbb{R}^+,
    \end{equation}
    otherwise we can just shift the profile along $y$-axis.
    
    It is important to note that for the perturbation of the traveling wave, we are actually allowed to take $\mathcal{P}(y)=\cos{l y}$ with $l\in\mathbb{R}^+$. However, only integer frequencies perturbation are reasonable for the radial case, since $\mathcal{P}(\theta)$ has to be a $2\pi$-periodic function. 
\subsection{Solutions after perturbation}
     Let $\Tilde{D}(t)$, enclosed by $\Tilde{\mathcal{B}}_t$, denote the tumoral region after the perturbation. Then the perturbed functions $(\Tilde{c},\Tilde{p},\Tilde{\rho})$ satisfy the equations (boundary conditions will be specified in the next subsection):
    \begin{subequations}
    \label{eqn:perturbed equations in general}
    \begin{alignat}{2}
    -\lap \Tilde{c} +\Psi(\Tilde{\rho},\Tilde{c})=0,\qquad&\text{on}\quad\mathbb{R}^2,\\
    -\lap \Tilde{p}=G_0 \Tilde{c},\qquad&\text{in}\quad\Tilde{D}(t),
    \end{alignat}
    recall that $\Psi(\Tilde{\rho},\Tilde{c})$ reads \eqref{eqn:infinite m invitro} in the \emph{in vitro} model and \eqref{eqn:infinite m invivo} in the \emph{in vivo} model. 
    \end{subequations}
   When the boundary perturbation vanishes, \eqref{eqn:perturbed equations in general} reduce to the the unperturbed problem, where the solutions are given in a closed-form. In the presence of the boundary perturbation, we still have $\Tilde{\rho}=\chi_{\Tilde{D}}$ since it remains as a patch, but the solution to $\Tilde{c}$ and $\Tilde{p}$ are no longer available. However, we can alternatively seek asymptotic solutions of  $\Tilde{c}$ and $\Tilde{p}$ with respect to $\delta$, while the condition $\Tilde{\rho}=\chi_{\Tilde{D}}$ help to linearize the calculation. We elaborate the  asymptotic analysis procedures as follows.
    
    Firstly, corresponding to the small perturbation \eqref{eqn:boundary perturbation}, we have the following asymtotic expansion with respect to the small value $\delta$:
    \begin{subequations}
    \label{eqn:delta asymtotic expansion}
    \begin{alignat}{2}
    \Tilde{c}(\zeta,\vartheta,t)&=c_0(\zeta,t)+\delta c_1(\zeta,\vartheta,t)+O(\delta^2),\\
    \Tilde{p}(\zeta,\vartheta,t)&=p_0(\zeta,t)+\delta p_1(\zeta,\vartheta,t)+O(\delta^2).
    \end{alignat}
    \end{subequations}
   Since the perturbation scale is assumed to be very small, i.e., $\delta\ll 1$, the behavior of the perturbed solutions are dominated by the unperturbed ones. Thus, the leading order terms $c_0(\zeta,t)$ and $p_0(\zeta,t)$ take the same expression as the solutions without perturbation, which have been solved in Section \ref{sec:Analytic solutions}. On the other hand, the main response corresponding to the perturbation are captured by the first-order terms $c_1(\zeta,\vartheta,t)$ and $p_1(\zeta,\vartheta,t)$. Note that besides variable $\zeta$ they depend on $\vartheta$ as well. 
   
   We continue to investigate the structures of $c_1$ and $p_1$ when the perturbation profile \eqref{eqn:boundary perturbation} is given by $\mathcal{P}(\vartheta)=\mathcal{P}_{l}(\vartheta)$, here $\mathcal{P}_{l}(\vartheta)$ presents \eqref{eqn:radial perturbation profile} or \eqref{eqn:TW perturbation profile} in the respective case. In either case, the perturbed tumoral region $\Tilde{D}$ still possess a symmetry, or periodicity, respect to the parameter $\vartheta$ ($\theta$ for the radial case and $y$ for the traveling wave case). Then we have following conclusion for the perturbed solutions $(\Tilde{c},\Tilde{p})$.
   \begin{lemma}
   If the perturbed solutions are unique, then they must process the same periodicity as the boundary geometry.
   \end{lemma}
   \begin{proof}
   For either scenario, the front of traveling wave or radially symmetric boundary, we assume the boundary has periodicity $\vartheta^*$. Then, with respect to \eqref{eqn:boundary perturbation} we have:
   \begin{equation}
       \zeta(\vartheta)=\zeta(\vartheta+\vartheta^*),
   \end{equation}
   where $\zeta(\vartheta)\defeq Z(t)+\delta\mathcal{P}(\vartheta)$. For $\mathcal{P}(\vartheta)=\cos{l\vartheta}$, $\vartheta^*$ is given by $\vartheta^*=\frac{2\pi}{l}$.
   We define the translation operator $\tau_{\vartheta^*}(\zeta,\vartheta):(\zeta,\vartheta)\mapsto (\zeta,\vartheta+\vartheta^*)$. One can easily observe that the nutrient equations (for either \emph{in vitro} or \emph{in vivo}) and the pressure equation are both invariant under $\tau_{\vartheta^*}$ since the operator simply corresponds to a translation or a rotation, and diffusion operator is isotropic. Moreover, the boundary geometry and boundary conditions remain the same under the operator $\tau_{\vartheta^*}$ as well. Thus, the uniqueness of the solution yield that the unique solutions $c^*$ and $p^*$ must possess the same periodicity as the boundary geometry. That is, 
   \begin{subequations}
    \begin{alignat}{2}
    c^*(\zeta,\vartheta)&=c^*(\tau_{\vartheta^*}(\zeta,\vartheta)),\\
    p^*(\zeta,\vartheta)&=p^*(\tau_{\vartheta^*}(\zeta,\vartheta)).
    \end{alignat}
    \end{subequations}
   \end{proof}
   According to the above lemma, to be consistent with the boundary's periodicity, we expand $c_1(\zeta, \vartheta,t)$ and $p_1(\zeta, \vartheta,t)$ as Fourier series, and \eqref{eqn:delta asymtotic expansion} can be further written as:
    \begin{subequations}
    \label{eqn:c&p general expansion}
    \begin{alignat}{2}
    \Tilde{c}(\zeta,\vartheta,t)&=c_0(\zeta,t)+\delta(t)\Sigma_{k=1}^{\infty} c_1^k(\zeta,t)\mathcal{P}_l^k(\vartheta)+O(\delta^2),\\
    \Tilde{p}(\zeta,\vartheta,t)&=p_0(\zeta,t)+\delta(t) \Sigma_{k=1}^{\infty}p_1^k(\zeta,t)\mathcal{P}_l^k(\vartheta)+O(\delta^2),
    \end{alignat}
    where $\mathcal{P}_l^k(\vartheta)=\cos{k l \vartheta}$.
    \end{subequations}
    In the above expansions, $c_1^k(\zeta,t)$ and $p_1^k(\zeta,t)$ (with $k\in\mathbb{N}^+$) serve as the Fourier coefficients with $O(1)$. From the calculation in the later sections (Section \ref{sec:TW The detailed calculations for the two nutrient regimes} and Section \ref{sec:Radial The detailed calculations for the two nutrient regimes}), we will see that only $c_1^1$ and $p_1^1$, the coefficients of the wave number that is the same as the perturbation, do not vanish. Therefore, it suffices to keep the first term in the series \eqref{eqn:c&p general expansion} and drop the superscript in $c_1^1$, $p_1^1$ and, $\mathcal{P}_l^1$ . Thus, \eqref{eqn:c&p general expansion} writes
    \begin{subequations}
    \label{eqn:c&p general expansion2}
    \begin{alignat}{2}
    \label{eqn:c general expansion2}
    \Tilde{c}(\zeta,\vartheta,t)&=c_0(\zeta,t)+\delta(t) c_1(\zeta,t)\mathcal{P}_l(\vartheta)+O(\delta^2),\\
    \label{eqn:p general expansion2}
    \Tilde{p}(\zeta,\vartheta,t)&=p_0(\zeta,t)+\delta(t) p_1(\zeta,t)\mathcal{P}_l(\vartheta)+O(\delta^2).
    \end{alignat}
    \end{subequations}
In the traveling wave case, the dependency of $t$ can be removed for the terms $c_j$ and $p_j$ ($j=\left\{0,1\right\}$), since the unperturbed tumor boundary do not evolve in time.   
Finally, by plugging the expansion \eqref{eqn:c&p general expansion2} into \eqref{eqn:perturbed equations in general} and collect the first order terms we get
\begin{subequations}
\label{eqn:general soln for first order terms}
\begin{alignat}{2}
\label{eqn:general soln for first order terms I}
-\lap (c_1(\zeta,t)\mathcal{P}_l(\vartheta)) +\lambda c_1(\zeta,t)\mathcal{P}_l(\vartheta)=0,\qquad&\text{in}\quad \Tilde{D}(t),\\
\label{eqn:general soln for first order terms II}
-\lap (p_1(\zeta,t)\mathcal{P}_l(\vartheta))=G_0 (c_1(\zeta,t)\mathcal{P}_l(\vartheta)),\qquad&\text{in}\quad \Tilde{D}(t),
\end{alignat}
for either nutrient regime. In addition, for the \emph{in vivo} model $c_1$ also satisfies 
\begin{equation}
\label{eqn:general soln for first order terms III}
    -\lap (c_1(\zeta,t)\mathcal{P}_l(\vartheta)) + c_1(\zeta,t)\mathcal{P}_l(\vartheta)=0,\qquad\text{in}\quad \mathbb{R}^2\setminus\Tilde{D}(t).
\end{equation}
\end{subequations}
where we used the fact that the zero order terms satisfy \eqref{eqn:perturbed equations in general}. By solving \eqref{eqn:general soln for first order terms}, one can get the solutions of $c_1$ and $p_1$ for the respective models. Note that \eqref{eqn:general soln for first order terms} implies the expression of $c_1$ and $p_1$ depend on the wave number $l$. The detailed computation will be carried out for the specific cases in the later sections. 
\subsection{Match the boundary condition}
In the last part of this section, we explain how to determine the particular solutions of $c_1$ and $p_1$ by matching the boundary conditions. We also show that by using the expression of $p_1$, one can determine the evolution of the perturbation magnitude.

In this section, we always assume the perturbation profile $\mathcal{P}(\vartheta)$ is given by $\mathcal{P}_l(\vartheta)$. And note that given $t$ for any $\vartheta\in\mathcal{R}$, $(Z+\delta\mathcal{P}_l(\vartheta),\vartheta)$ presents a point on the perturbed boundary $\Tilde{\mathcal{B}}_t$, which is originally at the position $(Z,\vartheta)\in\mathcal{B}_t$. Recall that $c_0$ and $c_1$ (similarly for $p_0$ and $p_1$) only depend on the variable $\zeta$ in space, and the unperturbed boundary $\mathcal{B}_t$ is characterized as the contour of $\zeta$ with level set index $Z(t)$ (see \eqref{eqn:general unperturbed boundary}).

Since the analytical solutions are not available for the perturbed problem, it is not practical to enforce the boundary conditions in the precise way. Instead, since we seek the first order perturbation solutions due to the boundary variation, we can approximately match the the perturbed solutions at the perturbed boundary up to $O(\delta^2)$ error with the their evaluations at the unperturbed boundary.

For the \emph{in vitro} model, the perturbed solution $\Tilde{c}$ satisfies the boundary condition:
\begin{equation}
\label{eqn: perturbed invitro bc}
  \Tilde{c}=c_B,\quad\text{at}\quad\Tilde{\mathcal{B}}_t.
\end{equation}
Thus by using expansion \eqref{eqn:c general expansion2}, we can evaluate $\Tilde{c}$ at the perturbed boundary point $(Z+\delta\mathcal{P}_l(\vartheta),\vartheta)$ to get
\begin{align}
\label{eqn:nutrient expansion}
\Tilde{c}(Z+\delta\mathcal{P}_l(\vartheta),\vartheta,t)
&=c_0(Z+\delta\mathcal{P}_l(\vartheta),t)+\delta c_1(Z+\delta\mathcal{P}_l(\vartheta),t)\mathcal{P}_l(\vartheta)+O(\delta^2)\\
&=c_0(Z,t)+\delta \partial_{\zeta}c_0(Z,t)\mathcal{P}_l(\vartheta)+\delta c_1(Z,t)\mathcal{P}_l(\vartheta)+O(\delta^2)\nonumber,
\end{align}
where we used the Taylor expansions for $c_0(Z+\delta\mathcal{P}_l,t)$ and $c_1(Z+\delta\mathcal{P}_l,t)$.
By the boundary conditions of the perturbed and unperturbed problems, we have $\Tilde{c}(Z+\delta\mathcal{P}_l(\vartheta),\vartheta,t)=c_0(Z,t)=c_B$ for arbitrary $\vartheta\in\mathcal{R}$. Thus the zero order terms in \eqref{eqn:nutrient expansion} are canceled out, and by balancing the first order terms in \eqref{eqn:nutrient expansion} we get
\begin{equation}
\label{eqn:invitro nutrient boundary}
    \partial_{\zeta}c_0(Z,t)+c_1(Z,t)=0.
\end{equation}
While for the \emph{in vivo} model, $\Tilde{c}$ and its normal derivative are both continuous at $\Tilde{\mathcal{B}}_t$. And in either kind of boundary, the normal derivative of $\Tilde{c}(\zeta,\vartheta,t)$ on $\Tilde{\mathcal{B}}_t$ is given by $\partial_{\zeta}\Tilde{c}(\zeta,\vartheta,t)$ for any $\Tilde{B}(\zeta,\vartheta)\in\Tilde{\mathcal{B}}_t$. And if we decompose the solution $\Tilde{c}$ according to the regions as $\Tilde{c}=\Tilde{c}^{\text{(i)}}\chi_{\Tilde{D}}+\Tilde{c}^{\text{(o)}}\chi_{\mathbb{R}^2\setminus\Tilde{D}}$, i.e., let $\Tilde{c}^{\text{(i)}}$ denotes the nutrient solution inside the tumor, and $\Tilde{c}^{\text{(o)}}$ presents the outside solution. Then, the continuity at the boundary yields
\begin{subequations}
\label{eqn:continuity of c at perturbed boundary}
\begin{alignat}{2}
\Tilde{c}^{\text{(i)}}(Z+\delta\mathcal{P}_l(\vartheta),\vartheta,t)=\Tilde{c}^{\text{(o)}}(Z+\delta\mathcal{P}_l(\vartheta),\vartheta,t),\qquad&\forall\vartheta\in\mathcal{R}\\
\partial_{\zeta}\Tilde{c}^{\text{(i)}}(Z+\delta\mathcal{P}_l(\vartheta),\vartheta,t)=\partial_{\zeta}\Tilde{c}^{\text{(o)}}(Z +\delta\mathcal{P}_l(\vartheta),\vartheta,t),\qquad&\forall\vartheta\in\mathcal{R},
\end{alignat}
\end{subequations}
With the same spirit of \eqref{eqn:nutrient expansion}, for $\partial_{\zeta}\Tilde{c}(Z+\delta\mathcal{P}_l(\vartheta),\vartheta,t)$ we have
\begin{equation}
\label{eqn:nutrient derivative expansion}
    \partial_{\zeta}\Tilde{c}(Z+\delta\mathcal{P}_l(\vartheta),\vartheta,t)=\partial_{\zeta}c_0(Z,t)+\partial^2_{\zeta}c_0(Z,t)\delta\mathcal{P}_l(\vartheta)+\partial_{\zeta}c_1(Z,t)\delta\mathcal{P}_l(\vartheta)+O(\delta^2).
\end{equation}
 Since $c_0$ is the solution to the unperturbed problem (given by \eqref{eqn:Radial nutrient steady soln invivo} or \eqref{eqn:TW nutrient steady soln invivo} for the respect case), $c_0$ and $c_0'$ are both continuous at the unperturbed boundary $\mathcal{B}_t=\left\{\zeta=Z(t)\right\}$. More precisely,
 \begin{align}
     c_0^{\text{(i)}}(Z,t)&=c_0^{\text{(o)}}(Z,t),\\
     \partial_{\zeta}c_0^{\text{(i)}}(Z,t)&=\partial_{\zeta}c_0^{\text{(o)}}(Z,t).
 \end{align}
 Thus, by using the expansions \eqref{eqn:nutrient expansion} and \eqref{eqn:nutrient derivative expansion}, the boundary condition \eqref{eqn:continuity of c at perturbed boundary} yields 
\begin{subequations}
\label{eqn:invivo nutrient boundary}
\begin{alignat}{2}
    c_1^{\text{(i)}}(Z,t)&=c_1^{\text{(o)}}(Z,t),\\
    \partial^2_{\zeta}c_0^{\text{(i)}}(Z,t)+\partial_{\zeta}c_1^{\text{(i)}}(Z,t)&=\partial^2_{\zeta}c_0^{\text{(o)}}(Z,t)+\partial_{\zeta}c_1^{\text{(o)}}(Z,t).
\end{alignat}
\end{subequations}
By using \eqref{eqn:invitro nutrient boundary} or \eqref{eqn:invivo nutrient boundary} as the boundary condition for $c_1$, we can work out the particular solution of $c_1$ in the respective cases. We mention that when the boundary is the traveling front, to carry out the full expression of $c_1$, we also need to use the boundary condition $c_1(-\infty,y)=c_1(+\infty,y)=0$ for any $y\in\mathbb{R}$, which is derived from $\Tilde{c}(-\infty,y)=0$ and $\Tilde{c}(+\infty,y)=c_B$ for any $y\in\mathbb{R}$. The detail calculations will be carried out for each specific case later (Section \ref{sec:TW The detailed calculations for the two nutrient regimes} and Section \ref{sec:Radial The detailed calculations for the two nutrient regimes}).

In either nutrient model, the perturbed pressure solution $\Tilde{p}$ satisfies the boundary condition 
\begin{equation}
\label{eqn:general perturbed pressure bc}
    \Tilde{p}=0,\quad\text{at}\quad\Tilde{\mathcal{B}}_t.
\end{equation}
Similar to the previous calculations. By using the expansion \eqref{eqn:p general expansion2} to evaluate $\Tilde{p}$ at $(Z+\delta\mathcal{P}_l(\vartheta))\in\Tilde{\mathcal{B}}_t$, we get
\begin{equation}
\label{eqn:pressure expansion}
    \Tilde{p}(Z+\delta\mathcal{P}_l(\vartheta),\vartheta,t)
=p_0(Z,t)+\partial_{\zeta}p_0(Z,t)\delta\mathcal{P}_l(\vartheta)+p_1(Z,t)\delta\mathcal{P}_l(\vartheta)+O(\delta^2).
\end{equation}
The perturbed and unperturbed boundary condition yield that $\Tilde{p}(Z+\delta\mathcal{P}_l(\vartheta),\vartheta,t)$ and $p_0(Z,t)$ both equal to zero. In particular, $p_0$ as the unperturbed solution has already been solved in the Section \ref{sec:Analytic solutions}. Thus we get 
\begin{equation}
\label{eqn:pressure boundary}
    \partial_{\zeta}p_0(Z,t)+p_1(Z,t)=0.
\end{equation}
Then by using the expression of $c_1$ (see \eqref{eqn:general soln for first order terms II}) and \eqref{eqn:pressure boundary}, we can further determine the particular solution of $p_1$. For the traveling wave case, we also use the condition $\partial_{\zeta}p_1(-\infty,y)=0$ for $\forall y\in\mathbb{R}$. Finally, the normal boundary speed \eqref{eqn:general boudary speed} yields:
\begin{equation}
\label{eqn:perturbed boundary speed}
    \frac{d (Z(t)+\delta(t)\mathcal{P}_l(\vartheta))}{dt}=-\partial_{\zeta}\Tilde{p}(Z+\delta\mathcal{P}_l(\vartheta),\vartheta,t).
\end{equation}
By plugging the expression of $\Tilde{p}$ into \eqref{eqn:perturbed boundary speed} and taking Taylor expansion for the $\zeta$ variable, we get
\begin{equation}
    \frac{d Z}{dt}+\frac{d\delta}{dt}\mathcal{P}_l(\vartheta)
    =-\left(\partial_{\zeta}p_0(Z,t)+\partial^2_{\zeta}p_0(Z,t)\delta\mathcal{P}_l(\vartheta)+\partial_{\zeta}p_1(Z,t)\delta\mathcal{P}_l(\vartheta)+O(\delta^2)\right).
\end{equation}
Since the unperturbed problem yields $\frac{d Z}{dt}=-\partial_{\zeta}p_0(Z,t)$, the above identity can be further simplified into
\begin{equation}
\label{eqn:change rate of delta}
    \delta^{-1}\frac{d\delta}{dt}=-\left(\partial^2_{\zeta}p_0(Z,t)+\partial_{\zeta}p_1(Z,t)+O(\delta)\right).
\end{equation}
In the end, we determine the evolution of the perturbation magnitude by the sign of $\delta^{-1}\frac{d\delta}{dt}$. If it is positive, then it implies the growing of the magnitude. For the radial case, $Z(t)$ is given by the unperturbed tumor radius $R(t)$, therefore $\delta^{-1}\frac{d\delta}{dt}\sim-(\partial^2_{r}p_0(R(t),t)+\partial_{r}p_1(R(t),t))$. Note that the leading order of $\delta^{-1}\frac{d\delta}{dt}$ is independent on $\theta$, which parameterize the boundaries. While, for the traveling wave case, $Z(t)=0$, thus $\delta^{-1}\frac{d\delta}{dt}\sim-(\partial^2_{\xi}p_0(0)+\partial_{\xi}p_1(0))$, which is independent of $y$. Furthermore, under the same nutrient regime, we expect that the boundary instability of the radius case will coincide with that of the traveling wave when $R$ increase to infinity.

\section{Stability of traveling waves in the two nutrient models}
\label{sec:traveling wave}
In this section, we study the boundary stability of the traveling wave front under two nutrient regimes. In Section \ref{sec:Set up for the traveling wave perturbation}, we establish the set up and main conclusions. In Section \ref{sec:TW The detailed calculations for the two nutrient regimes}, we work out the expression of $\delta^{-1}\frac{d\delta}{d t}$ for the two nutrient models, which serves as the proof of Theorem \ref{thm:main thm for TW}. And in Section \ref{sec:Boundary stability analysis for the two nutrient models TW}, we prove the mathematical properties of $\delta^{-1}\frac{d\delta}{d t}$ summarized in Corollary \ref{cor:TW}, and these properties further yield the boundary behaviors summarized in Remark \ref{rmk:TW}.

\subsection{Setup and main results}
\label{sec:Set up for the traveling wave perturbation}
As presented in Section \ref{sec:Perturbation of the boundary}, in the traveling wave case we employ the Euler coordinate system $(\xi,y)$.  In the absence of perturbation, the tumor boundary is defined by \eqref{eqn:traveling wave boundary} with the level set index $Z(t)=0$, and the tumor region is the left half space $D(t)=\left\{(\xi,y)\vert\xi\leq 0\right\}$. Then following the framework of Section \ref{sec:Framework of perturbation analysis}, we consider the perturbation by a single wave mode:
\begin{equation}
\mathcal{P}_l(y)=\cos{ly}\qquad\text{with}\quad l\in\mathbb{R}^+,
\end{equation}
thus the perturbed boundary \eqref{eqn:boundary perturbation} writes:
    \begin{equation}
    \label{eqn:TW boundary perturbation}
        \Tilde{\mathcal{B}}_t(y)=\left\{(\xi,y)\vert\xi=\delta(t) \cos{ly}, y\in\mathbb{R}\right\},
    \end{equation}
and the perturbed tumor region becomes
    \begin{equation}
    \label{eqn:TW perturbed tumor region}
        \Tilde{D}(t)=\left\{(\xi,y)\vert\xi\leq\delta(t) \cos{ly}, y\in\mathbb{R}\right\}.
    \end{equation}
Then correspond to the above perturbation, the perturbed solutions $c$ and $p$ solves \eqref{eqn:perturbed equations in general}. Note that we dropped the tilde of the perturbed solutions for simplicity. Further more, the perturbed solutions possess the asymptotic expansions:
   \begin{subequations}
    \label{eqn:TW delta asymtotic expansion}
    \begin{alignat}{2}
    \label{eqn:TW delta asymtotic expansion c1}
    c(\xi,y,t)&=c_0(\xi)+\delta(t) c_1(\xi,y)+O(\delta^2),\\
    \label{eqn:TW delta asymtotic expansion p1}
    p(\xi,y,t)&=p_0(\xi)+\delta(t) p_1(\xi,y)+O(\delta^2),
    \end{alignat}
    where the leading order terms $c_0$ and $p_0$ correspond to the solution of the unperturbed problems, which have been solved in Section \ref{sec:traveling plane solution for the invitro model} and Section \ref{sec:traveling plane solution for the invivo model} for the respective nutrient model. And the first order terms $c_1(\xi,y,t)$ and $p_1(\xi,y,t)$ can be further expanded as
    \begin{alignat}{2}
    \label{eqn:TW expansion of c1}
    c_1(\xi,y)=\Sigma_{k=1}^{\infty} c_1^k(\xi)\cos{kly},\\
    \label{eqn:TW expansion of p1}
    p_1(\xi,y)=\Sigma_{k=1}^{\infty} p_1^k(\xi)\cos{kly}.
    \end{alignat}
    with $l\in\mathbb{R}^+$, so that $c_1$ has the same periodicity as the boundary geometry.
    \end{subequations}
However, from the calculation later we will see that $c_1^k(\xi)=p_1^k(\xi)=0$ for any $k\neq 1$. 

For the \emph{in vivo} model, we use the superscript $\text{(i)}$ or $\text{(o)}$ to denote the solution inside or outside the tumor region $\Tilde{D}(t)$. Then according to \eqref{eqn:general soln for first order terms I} and \eqref{eqn:general soln for first order terms III} we have
\begin{subequations}
\label{eqn:TW invivo c1 eqns}
\begin{alignat}{2}
-\lap c_1^{\text{(i)}}(\xi,y)+\lambda c_1^{\text{(i)}}(\xi,y)=0,\\
-\lap c_1^{\text{(o)}}(\xi,y)+c_1^{\text{(o)}}(\xi,y)=0.
\end{alignat}
\end{subequations}
And by using the expansion in \eqref{eqn:nutrient expansion} and \eqref{eqn:nutrient derivative expansion}, the series form of $c_1(\xi,y)$ in $\eqref{eqn:TW expansion of c1}$, the boundary condition \eqref{eqn:continuity of c at perturbed boundary} yields
\begin{subequations}
\begin{alignat}{2}
\label{eqn:TW invivo cts at tumor boundary a}
    c_1^{\text{(i)},k}(0)&=c_1^{\text{(o)},k}(0),\quad\forall k\in\mathbb{N}^+,\\
\label{eqn:TW invivo cts at tumor boundary b}
    \partial_{\xi}c_1^{\text{(i)},k}(0)&=\partial_{\xi}c_1^{\text{(o)},k}(0),\quad\forall k\geq 2,\\
\label{eqn:TW invivo cts at tumor boundary c}
    \partial^2_{\xi}c_0^{\text{(i)}}(0)+\partial_{\xi}c_1^{\text{(i)},1}(0)&=\partial^2_{\xi}c_0^{\text{(o)}}(0)+\partial_{\xi}c_1^{\text{(o)},1}(0).
\end{alignat}
Further more, the assumptions $c(-\infty,y)=0$ and $c(+\infty,y)=c_B$ for any $y\in\mathbb{R}$ yields
\begin{alignat}{2}
\label{eqn:TW c_1 -infty}
    c_1^{\text{(i)},k}(-\infty)=0,\\
\label{eqn:TW c_1 +infty}
    c_1^{\text{(o)},k}(+\infty)=0,
\end{alignat}
for any $k\in\mathbb{N}^+$.
\end{subequations}

For the \emph{in vitro} model, $c$ presents the nutrient solution inside the tumor and equation \eqref{eqn:general soln for first order terms I} writes
\begin{equation}
\label{eqn:TW c1 eqns invitro}
    -\lap c_1(\xi,y)+\lambda c_1(\xi,y)=0,\quad\text{in}\quad \Tilde{D}(t).
\end{equation}
By using \eqref{eqn:nutrient expansion} and the series expansion of $c_1$ in \eqref{eqn:TW expansion of c1}, the boundary condition \eqref{eqn: perturbed invitro bc} yields:
\begin{subequations}
\label{eqn:TW invitro boundary condtion}
\begin{alignat}{2}
\partial_{\zeta}c_0(0)+c_1^1(0)=0.\\
c_1^{k}(0)=0,\quad\forall k\geq 2.
\end{alignat}
and similar to the \emph{in vivo} model, the assumption $c(-\infty,y)=0$ for any $y\in\mathbb{R}$ gives
\begin{equation}
\label{eqn:TW invitro -infty c1}
    c_1^k(-\infty)=0,\qquad\forall k\in\mathbb{N}^+.
\end{equation}
\end{subequations}

Once $c_1(\xi,y)$ is determined, we can move on to the study of the first order term of pressure, i.e., $p_1(\xi,y)$. Under either nutrient regime, $p_1(\xi,y)$ satisfies the equation:
\begin{equation}
\label{eqn:TW p1 equation}
    -\lap p_1(\xi,y)=G_0 c_1(\xi,y),\quad\text{in}\quad\Tilde{D}(t).
\end{equation}
By using the expansion \eqref{eqn:pressure expansion}, the series form of $p_1$ in \eqref{eqn:TW expansion of p1}, the boundary condition \eqref{eqn:general perturbed pressure bc} yields
\begin{subequations}
\label{eqn:TW p1 boundary conditions}
\begin{alignat}{2}
\partial_{\zeta}p_0(0)+p_1^1(0)=0,\\
p_1^{k}(0)=0,\quad\forall k\geq 2.
\end{alignat}
On the other hand, for the traveling wave case we require $\partial_{\zeta}p(-\infty,y)=0$, which further yields $\partial_{\xi}p_1(-\infty,y)=0$. Therefore,
\begin{equation}
    \partial_{\xi}p_1^k(-\infty)=0,\qquad\text{for}\quad\forall k\in\mathbb{N}^+.
\end{equation}
\end{subequations}
Once the expression of $p_1(\xi,y)$ is determined, we can further work out the expression of $\delta^{-1}\frac{d\delta}{d t}$ for the two nutrient regimes, which determines the evolution of the perturbation amplitude. Now we establish the main conclusions, and the details of the calculation will be left to the next subsection. 
\begin{theorem}
\label{thm:main thm for TW}
Given growing rate $G_0>0$, background concentration $c_B>0$, nutrient consumption rate $\lambda>0$, and perturbation frequency $l\in\mathbb{R}^+$. The perturbation evolution function, $\delta^{-1}\frac{d\delta}{d t}$, of the \emph{in vitro} model is given by:
\begin{equation}
\label{eqn:TW invitro evolution fcn}
\delta^{-1}\frac{d\delta}{dt}=\frac{G_0 c_B}{\sqrt{\lambda}}\cdot(\sqrt{\lambda}-\sqrt{\lambda+l^2})+O(\delta)\defeq F_1(\lambda,l)+O(\delta).
\end{equation}
For the \emph{in vivo} model, $\delta^{-1}\frac{d\delta}{d t}$ is given by:
\begin{align}
\label{eqn:TW invivo evolution fcn}
\delta^{-1}\frac{d\delta}{dt}
&=\frac{G_0 c_B}{\sqrt{\lambda}}\left(\frac{\sqrt{\lambda}-l}{\sqrt{\lambda}+1}+\frac{l-\sqrt{\lambda+l^2}}{\sqrt{\lambda+l^2}+\sqrt{1+l^2}}\right)+O(\delta)\\
&\defeq F_2(\lambda,l)+O(\delta).\nonumber
\end{align}
\end{theorem}
\begin{figure}
\label{fig:F2}
  \begin{center}
    \includegraphics[width=0.48\textwidth]{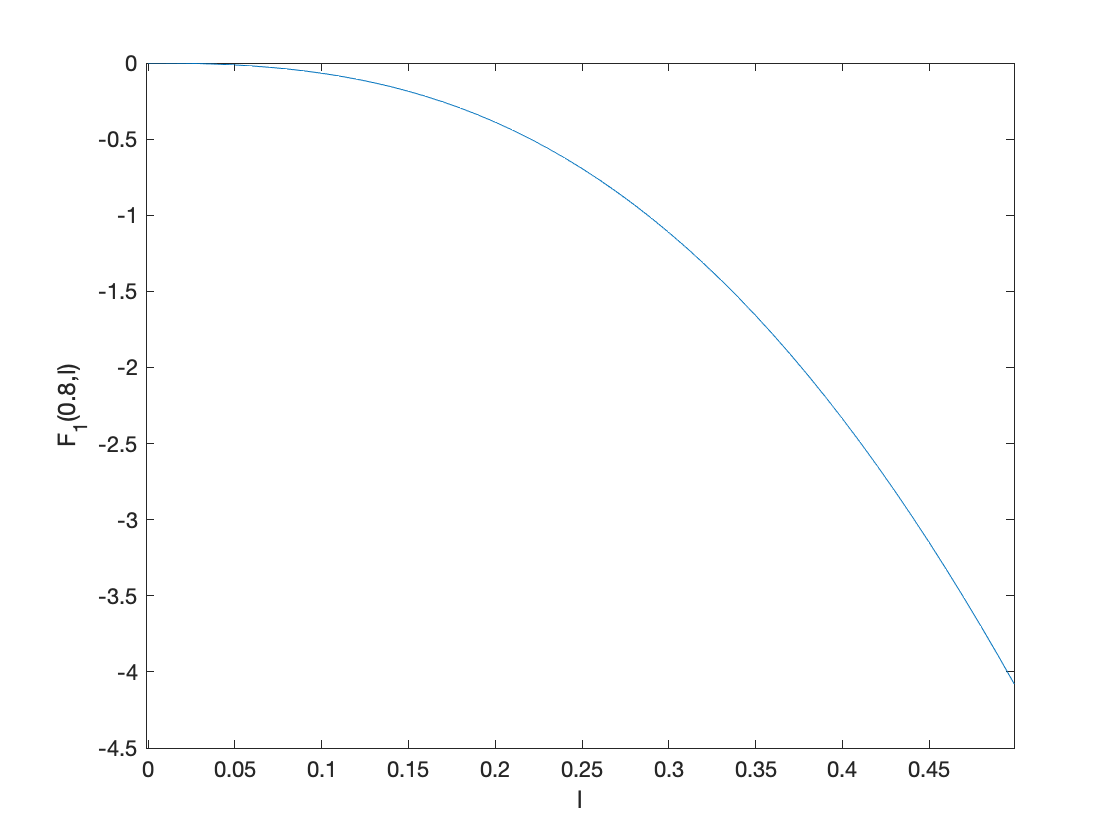}
    \includegraphics[width=0.48\textwidth]{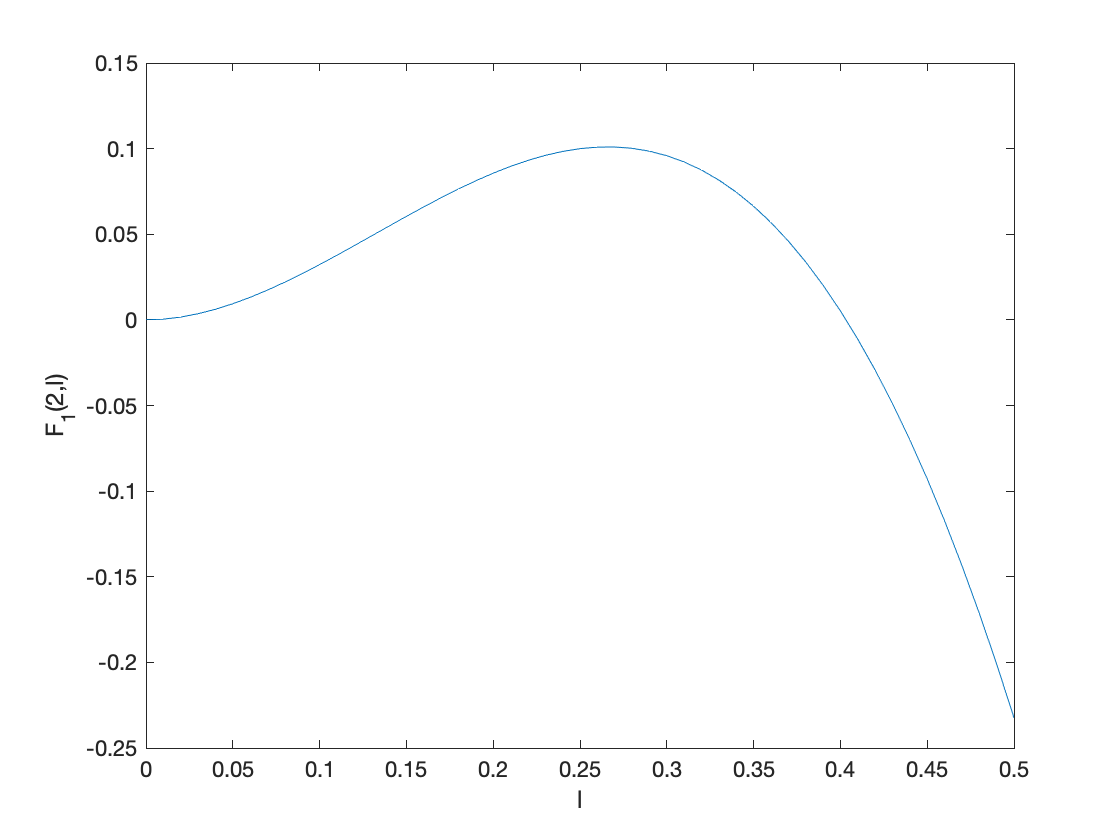}
    \caption{$F_2(\lambda,l)$ for $G_0=1$, $c_B=100$, left one: $\lambda=0.8$, right one: $\lambda=2$.}
    \label{fig:TW invivo}
  \end{center}
\end{figure}
Note that in either nutrient regime, the value of $G_0,c_B>0$ serve as scaling parameters, therefore do not influence the quantitative behavior of $\delta^{-1}\frac{d\delta}{dt}$. For the
\emph{in vitro} model, one can easily check that $F_1(\lambda,l)$ is always negative. For the \emph{in vivo} model, $F_2(\lambda, l)$ is plotted in Figure \ref{fig:TW invivo} for different choice of $\lambda$. Base on the expression of $F_1$, $F_2$, and the observations from Figure \ref{fig:TW invivo}, we prove following mathematical properties for the evolution equations.

\begin{corollary}
\label{cor:TW}
Fix $G_0>0$ and $c_B>0$. For any $\lambda>0$ and $l>0$, $F_1(\lambda,l)$ in \eqref{eqn:TW invitro evolution fcn} is always negative, therefore the perturbation amplitude always decreases in the \emph{in vitro} model. 
For the \emph{in vivo} model, $\delta^{-1}\frac{d\delta}{dt}$ is given by $F_2(\lambda,l)$ as in \eqref{eqn:TW invivo evolution fcn}. When $l$ approaches zero, $F_2$ has the asymptote:
\begin{subequations}
\begin{alignat}{2}
\label{eqn:F_2 asymp 0}
F_2(\lambda,l)\sim\frac{(\lambda-1)\cdot l^2+O(l^3)}{2\lambda(\sqrt{\lambda+1})(\sqrt{\lambda+l^2}+\sqrt{1+l^2})},\\
\intertext{and the limit at infinity:}
\label{eqn:F_2 asymp infty}
\lim_{l\rightarrow+\infty}F_2(\lambda,l)\rightarrow-\infty.
\end{alignat}
\end{subequations}
Further more, for $\lambda>1$ there exists $L>0$ such that $F_2(\lambda,l)>0$ for $l\in(0,L)$. And for $0<\lambda\leq 1$, we have $F_2(\lambda,l)<0$ for any $l>0$. And in particular, when $\lambda=1$, $F_2$ can be further simplified into:
\begin{equation}
\label{eqn:F1 lambda=1}
    F_2(1,l)=\frac{l(1-\sqrt{1+l^2})}{2\sqrt{1+l^2}}<0,\qquad\text{for}\quad l>0.
\end{equation}
\begin{remark}
The mathematical properties of $F_2(\lambda,l)$ in Corollary \ref{cor:TW} imply the following boundary behaviors:
\label{rmk:TW}
\begin{enumerate}
    \item  When the consumption rate is relatively large, the amplitude of low-frequency perturbations can grow in time, and the boundary propagation, therefore, becomes unstable. However, the amplitude of high-frequency perturbation decays. Correspondingly, the perturbed, wave like, boundary degenerates to the vertical line. 
    \item When the nutrient consumption rate is relatively small, the perturbation amplitude decreases for perturbation of any frequency, i.e., the wave like boundary always evolve to a vertical line in this regime.
\end{enumerate}
\end{remark}
\begin{remark}
For either nutrient model, $\delta^{-1}\frac{d\delta}{dt}\rightarrow 0$ as $l\rightarrow 0$, we claim that this relate to the single wave perturbation of the radially symmetric solution as its radius $R$ goes to infinity. This relationship will be further discussed in Section \ref{sec:Relationship between the Radial boundary and the traveling wave boundary}.
\end{remark}
\end{corollary}
\subsection{The detailed calculations for the two nutrient regimes}
\label{sec:TW The detailed calculations for the two nutrient regimes}
In this subsection we work out the expression of $\delta^{-1}\frac{d \delta}{dt}$ in Theorem \ref{thm:main thm for TW} for the two nutrient models.

For the \emph{in vivo} case. Plugging the expansion \eqref{eqn:TW expansion of c1} of $c_1(\xi,y)$ into \eqref{eqn:TW invivo c1 eqns}, together with the conditions \eqref{eqn:TW c_1 -infty} and \eqref{eqn:TW c_1 +infty}, for any $k\in\mathbb{N}^+$ we have:
\begin{subequations}
\begin{alignat}{2}
    c^{\text{(i)},k}_1(\xi)&=a_1^k e^{\sqrt{\lambda+k^2}\xi}\qquad&\text{for}\quad\xi\leq 0,\\
    c^{\text{(o)},k}_1(\xi)&=b_1^k e^{-\sqrt{1+k^2}\xi}\qquad&\text{for}\quad\xi\geq 0.
\end{alignat}
\end{subequations}
Recall that the leading order terms $c_0^{\text{(i)}}(\xi)$ and $c_0^{\text{(o)}}(\xi)$ are given by \eqref{eqn:TW nutrient steady soln invivo}. Then \eqref{eqn:TW invivo cts at tumor boundary a}-\eqref{eqn:TW invivo cts at tumor boundary c} yield $a_1^k=b_1^k=0$ for any $k\neq 1$, for $k=1$ we get nontrivial solution:
\begin{equation}
    c^{\text{(i)},1}_1(\xi)=c^{\text{(o)},1}_1(\xi)=-\frac{\sqrt{\lambda}c_B}{\sqrt{\lambda+l^2}+\sqrt{1+l^2}}e^{\sqrt{\lambda+l^2}\xi}.
\end{equation}
By now, $c_1^{\text{(i)},k}(\xi)$ and $c_1^{\text{(o)},k}(\xi)$ are determined for any $k$. Therefore, $c_1(\xi,y)$ is determined. Then by solving equation \eqref{eqn:TW p1 equation} together with boundary conditions \eqref{eqn:TW p1 boundary conditions} (with $p_0$ given by \eqref{eqn:TW pressure steady soln invivo}), we get $p_1^k(\xi)=0$ for any $k\neq 1$, and:
\begin{subequations}
\begin{equation}
      p_1^1(\xi)=Ae^{l \xi}-\frac{G_0}{\lambda}c^{\text{(i)},1}_1(\xi)=Ae^{l \xi}+\frac{G_0 c_B}{\sqrt{\lambda}(\sqrt{\lambda+l^2}+\sqrt{1+l^2})}e^{\sqrt{\lambda+l^2}\xi},
\end{equation}
with $A$ given by:
\begin{equation}
 A=\frac{G_0 c_B}{\sqrt{\lambda}}\left(\frac{1}{\sqrt{\lambda}+1}-\frac{1}{\sqrt{\lambda+l^2}+\sqrt{1+l^2}}\right).  
\end{equation}
\end{subequations}
By using the expression of $p_0(\xi)$ and $p_1^1(\xi)$, \eqref{eqn:change rate of delta} yields that up to an error of $O(\delta)$:
\begin{align}
\label{eqn:TW invivo growing rate of delta}
    \delta^{-1}\frac{d\delta}{dt}
    &=-\left(\partial^2_{\xi}p_0(0)+\partial_{\xi}p_1^1(0)\right)\\\nonumber
    &=\frac{G_0 c_B}{\sqrt{\lambda}}\left(\frac{\sqrt{\lambda}-l}{\sqrt{\lambda}+1}+\frac{l-\sqrt{\lambda+l^2}}{\sqrt{\lambda+l^2}+\sqrt{1+l^2}}\right)=F_2(\lambda,l).\nonumber
\end{align}
For the \emph{in vitro} model. Plugging the expansion of $c_1(\xi,y)$ \eqref{eqn:TW expansion of c1} into \eqref{eqn:TW c1 eqns invitro}, together with the conditions \eqref{eqn:TW invitro -infty c1}, for any $k\in\mathbb{N}^+$ we have:
\begin{equation}
    c^{k}_1(\xi)=a_1^k e^{\sqrt{\lambda+k^2}\xi}\qquad\text{for}\quad\xi\leq 0.
\end{equation}
And the leading order term $c_0(\xi)$ for this case is given by \eqref{eqn:TW nutrient steady soln invitro}. Then by using boundary condition \eqref{eqn:TW invitro boundary condtion}, we get $c^k_1(\xi)=0$ for any $k\neq 1$, and for $k=1$:
\begin{equation}
    c^1_1(\xi)=-c_B\sqrt{\lambda}e^{\sqrt{\lambda+l^2}\xi}.
\end{equation}
Then similar to the previous case, by solving equation \eqref{eqn:TW p1 equation} together with boundary conditions \eqref{eqn:TW p1 boundary conditions} (with $p_0$ given by \eqref{eqn:TW pressure steady soln invitro}), we get $p_1^k(\xi)=0$ for any $k\neq 1$. And for $k=1$:
\begin{equation}
     p^1_1(\xi)=\frac{G_0 c_B}{\sqrt{\lambda}}e^{\sqrt{\lambda+l^2}\xi}.
\end{equation}
Finally, by using the expression of $p_0$ and $p_1^1$, \eqref{eqn:change rate of delta} yields up to an error of $O(\delta)$:
\begin{align}
\label{eqn:TW invitro growing rate of delta}
    \delta^{-1}\frac{d\delta}{dt}
    &= -\left(\partial^2_{\xi}p_0(0)+\partial_{\xi}p_1^1(0)\right)\\\nonumber
    &=\frac{G_0 c_B}{\sqrt{\lambda}}(\sqrt{\lambda}-\sqrt{\lambda+l^2})=F_1(\lambda,l).
\end{align}
By now we complete the proof of Theorem \ref{thm:main thm for TW}.
\subsection{Boundary stability analysis for the two nutrient models}
\label{sec:Boundary stability analysis for the two nutrient models TW}
In this subsection, we prove the mathematical properties of $F_1$ and $F_2$ in Corollary \ref{cor:TW}, which further yields the boundary behaviors in Remark \ref{rmk:TW}. 

For the \emph{in vitro} model, one can easily check that $F_1(\lambda,l)\leq 0$ for any frequency $l\geq 0$. Thus, the amplitude of the perturbation decays as time evolves for any single frequency perturbation. Therefore, the proof of the argument for the \emph{in vitro} model in Corollary \ref{cor:TW} is completed.

For the \emph{in vivo} model, $\delta^{-1}\frac{d\delta}{d t}$ is given by \eqref{eqn:TW invivo evolution fcn} and plotted in Figure \ref{fig:TW invivo}. The limit \eqref{eqn:F_2 asymp infty} is obviously for checking. For the asymptote of $l$ approaches $0$, note that
\begin{align}
\label{eqn:rewrite F2}
    F_2(\lambda,l)
    &=\frac{(\sqrt{\lambda}-l)(\sqrt{\lambda+l^2}+\sqrt{1+l^2})+(l-\sqrt{\lambda+l^2})(\sqrt{\lambda}+1)}{\sqrt{\lambda}(\sqrt{\lambda}+1)(\sqrt{\lambda+l^2}+\sqrt{1+l^2})}\\
    &=\frac{N(\lambda,l)}{2\lambda(\sqrt{\lambda}+1)(\sqrt{\lambda+l^2}+\sqrt{1+l^2})}.\nonumber
\end{align}
By using the Taylor expansion $\sqrt{1+x}=1+\frac{x}{2}+O(x^2)$, when $l$ approaches to $0$ we have
\begin{equation}
\label{eqn:asymp N}
    N(\lambda,l)=\frac{\lambda-1}{2\sqrt{\lambda}}l^2+O(l^3).
\end{equation}
Therefore, if $\lambda>1$ then $F_2(\lambda,l)>0$ for $l$ close to zero, and $F_2(\lambda,l)<0$ for $l$ sufficiently large. Thus, by the intermediate value theorem and the continuity of $F_2(\lambda,l)$ in $l$, there exists $L>0$ such that $F_2(\lambda,l)>0$ for $l\in(0,L)$.

Now, we prove $F_2(\lambda,l)<0$ for any $0<\lambda\leq 1$ and $l>0$. From the expression in \eqref{eqn:TW invivo evolution fcn}, $F_2(\lambda,l)<0$ hold obviously for $l\geq\sqrt{\lambda}$. On the other hand, the denominator in \eqref{eqn:rewrite F2} is always positive. Therefore, it is sufficient for us to show the numerator $N(\lambda,l)$ is negative for $0<l<\sqrt{\lambda}\leq 1$. Indeed, by taking the derivative of $N(\lambda,l)$ with respect to $l$, we get
\begin{align}
    \frac{\partial N(\lambda,l)}{\partial l}
    &=-(\sqrt{\lambda+l^2}+\sqrt{1+l^2})+(\sqrt{\lambda}+1)+\frac{-l(l+1)}{\sqrt{\lambda+l^2}}+\frac{l(\sqrt{\lambda}-l)}{\sqrt{1+l^2}}\\\nonumber
    &\leq-\frac{l(l+1)}{\sqrt{\lambda+l^2}}+\frac{l(\sqrt{\lambda}-l)}{\sqrt{1+l^2}}\\\nonumber
    &\leq\frac{l(\sqrt{\lambda}-1-2l)}{\sqrt{\lambda+l^2}}<0.\nonumber
\end{align}
for $0<l<\sqrt{\lambda}\leq 1$. Finally, combine with the fact $N(\lambda,0)=0$, we can conclude $N(\lambda,l)< 0$ for $0< l<\sqrt{\lambda}\leq 1$. By now, we complete the proof of Corollary \ref{cor:TW}.
\section{Stability of radially symmetric boundary in the two nutrient models}
\label{sec:radial boundary}
In this section, we study the boundary stability of the radially symmetric solution under the two nutrient regimes. The structure of this section is arranged as follow. We establish the setups and main conclusions  in Section \ref{sec:Set up for the radial boundary perturbation}.  After that we carry out the calculation of $\delta^{-1}\frac{d\delta}{d t}$ for the two nutrient models in Section \ref{sec:Radial The detailed calculations for the two nutrient regimes}, which serves as the proof of Theorem \ref{thm:main thm for Radial}. Then, we proof the mathematical properties of the perturbation evolution functions summarized in Corollary \ref{cor:Radial} in Section \ref{sec:Radial Boundary stability analysis for the two nutrient models}. Finally, we discuss the relationship between the perturbation of the radial boundary and the traveling wave boundary in Section \ref{sec:Relationship between the Radial boundary and the traveling wave boundary}.

\subsection{Setup and main results}
\label{sec:Set up for the radial boundary perturbation}
For the radial case, we employ the polar coordinate system $(r,\theta)$. Before the perturbation, the tumor boundary is defined by \eqref{eqn:traveling wave boundary} with the level set index $Z(t)=R(t)$, and the tumor region corresponds to the disk with radius $R(t)$, $D(t)=\left\{(r,\theta)\vert r\leq R(t)\right\}$. Following the framework of Section \ref{sec:Framework of perturbation analysis}, we consider the perturbation by a single wave mode, i.e. $\mathcal{P}_l(\theta)=\cos{l\theta}$ with $\theta\in[-\pi, \pi)$. Then, the perturbed boundary \eqref{eqn:boundary perturbation} writes:
    \begin{equation}
    \label{eqn:Radial boundary perturbation}
        \Tilde{\mathcal{B}}_t(\theta)=\left\{(r,\theta)\vert r=R(t)+\delta(t)\cos{l\theta}, \theta\in[-\pi, \pi)\right\}.
    \end{equation}
Then the perturbed solutions (drop the tilde) $c$ and $p$ solves \eqref{eqn:perturbed equations in general}, with the perturbed tumor region $\Tilde{D}(t)$ enclosed by $\Tilde{\mathcal{B}}_t(\theta)$. Further more, $c$ and $p$ have the asymptotic expansions:
   \begin{subequations}
    \label{eqn:Radial delta asymtotic expansion}
    \begin{alignat}{2}
    \label{eqn:Radial delta asymtotic expansion c1}
    c(r,\theta,t)&=c_0(r,t)+\delta(t) c_1(r,\theta,t)+O(\delta^2),\\
    \label{eqn:Radial delta asymtotic expansion p1}
    p(r,\theta,t)&=p_0(r,t)+\delta(t) p_1(r,\theta,t)+O(\delta^2),
    \end{alignat}
    where the leading order terms $c_0$ and $p_0$ correspond to the unperturbed solutions for the respective nutrient model solved in Section \ref{sec:2D radially symmetric model for the in vitro model} and Section \ref{sec:2D radially symmetric model for the in vivo model}. And the first order terms $c_1(r,\theta,t)$ and $p_1(r,\theta,t)$ can be further expanded as
    \begin{alignat}{2}
    \label{eqn:Radial expansion of c1}
    c_1(r,\theta,t)=\Sigma_{k=1}^{\infty} c_1^k(r,t)\cos{kl\theta},\\
    \label{eqn:Radial expansion of p1}
    p_1(r,\theta,t)=\Sigma_{k=1}^{\infty} p_1^k(r,t)\cos{kl\theta}.
    \end{alignat}
    where $l$ as the perturbation wave number is a positive integer.
    \end{subequations}
And from the calculation we will verify that $c_1^k(r,t)=p_1^k(r,t)=0$ for any $k\neq 1$.

For the \emph{in vivo} model, same as before we use the superscript $\text{(i)}$ or $\text{(o)}$ to denote the solution inside or outside the tumor region $\Tilde{D}(t)$. Then according to \eqref{eqn:general soln for first order terms I} and \eqref{eqn:general soln for first order terms III} we have
\begin{subequations}
\label{eqn:Radial invivo c1 eqns}
\begin{alignat}{2}
-\lap c_1^{\text{(i)}}(r,\theta,t)+\lambda c_1^{\text{(i)}}(r,\theta,t)=0,\\
-\lap c_1^{\text{(o)}}(r,\theta,t)+c_1^{\text{(o)}}(r,\theta,t)=0.
\end{alignat}
and the Laplacian operator writes: $\lap=\frac{\partial^2}{\partial r^2}+\frac{1}{r}\frac{\partial}{\partial r}+\frac{1}{r^2}\frac{\partial^2}{\partial\theta^2}$.
\end{subequations}
By using the expansion in \eqref{eqn:nutrient expansion} and \eqref{eqn:nutrient derivative expansion}, the series form of $c_1(\xi,y)$ in $\eqref{eqn:Radial expansion of c1}$, the boundary condition \eqref{eqn:continuity of c at perturbed boundary} yields
\begin{subequations}
\begin{alignat}{2}
\label{eqn:Radial invivo cts at tumor boundary a}
    c_1^{\text{(i)},k}(R(t),t)&=c_1^{\text{(o)},k}(R(t),t),\quad\forall k\in\mathbb{N}^+,\\
\label{eqn:Radial invivo cts at tumor boundary b}
    \partial_{r}c_1^{\text{(i)},k}(R(t),t)&=\partial_{r}c_1^{\text{(o)},k}(R(t),t),\quad\forall k\geq 2,\\
\label{eqn:Radial invivo cts at tumor boundary c}
    \partial^2_{r}c_0^{\text{(i)}}(R(t),t)+\partial_{r}c_1^{\text{(i)},1}(R(t),t)&=\partial^2_{r}c_0^{\text{(o)}}(R(t),t)+\partial_{r}c_1^{\text{(o)},1}(R(t),t).
\end{alignat}
In addition, the assumptions $c(0,\theta)<+\infty$ and $c(+\infty,\theta)=c_B$ for any $\theta\in[-\pi, \pi)$ yield that
\begin{alignat}{2}
\label{eqn:Radial c_1 -infty}
    c_1^{\text{(i)},k}(0,t)<+\infty,\\
\label{eqn:Radial c_1 +infty}
    c_1^{\text{(o)},k}(+\infty,t)=0,
\end{alignat}
for any $k\in\mathbb{N}^+$.
\end{subequations}

For the \emph{in vitro} model, $c$ presents the nutrient solution inside the tumor and equation \eqref{eqn:general soln for first order terms I} writes
\begin{equation}
\label{eqn:Radial c1 eqns invitro}
    -\lap c_1(r,\theta,t)+\lambda c_1(r,\theta,t)=0,\quad\text{in}\quad \Tilde{D}(t).
\end{equation}
And by using \eqref{eqn:nutrient expansion} and the series expansion of $c_1$ in \eqref{eqn:Radial expansion of c1}, the boundary condition \eqref{eqn: perturbed invitro bc} yields:
\begin{subequations}
\label{eqn:Radial invitro boundary condtion}
\begin{alignat}{2}
\partial_{\zeta}c_0(R(t),t)+c_1^1(R(t),t)=0.\\
c_1^{k}(R(t),t)=0,\quad\forall k\geq 2.
\end{alignat}
Similar to the \emph{in vivo} model, the boundary condition \eqref{eqn:Radial c_1 -infty} remains ture (drop the superscript
$\text{(i)}$).
\end{subequations}

Once $c_1(r,\theta,t)$ is determined by the boundary value problems above, we can further determine $p_1(r,\theta,t)$ for the corresponding model. Under either nutrient regime, $p_1(r,\theta,t)$ satisfies the equation:
\begin{equation}
\label{eqn:Radial p1 equation}
    -\lap p_1(r,\theta,t)=G_0 c_1(r,\theta,t),\quad\text{in}\quad\Tilde{D}(t).
\end{equation}
By using the expansion \eqref{eqn:pressure expansion}, the series form of $p_1$ in \eqref{eqn:Radial expansion of p1}, the boundary condition \eqref{eqn:general perturbed pressure bc} yields
\begin{subequations}
\label{eqn:Radial p1 boundary conditions}
\begin{alignat}{2}
\partial_{r}p_0(R(t),t)+p_1^1(R(t),t)=0.\\
p_1^{k}(R(t),t)=0,\quad\forall k\geq 2.
\end{alignat}
And by asymmetry we also have $\partial_{r}p(0,\theta,t)=0$, which further provides
\begin{equation}
    \partial_{r}p_1^k(0,t)=0,\qquad\text{for}\quad\forall k\in\mathbb{N}^+.
\end{equation}
\end{subequations}
Once $p_1(r,\theta,t)$ are determined, we can work out the expression of $\delta^{-1}\frac{d\delta}{d t}$ for the two nutrient regimes as in \eqref{eqn:change rate of delta}, which further determines the evolution of the perturbation amplitude. Now we establish the main conclusions, and the detailed calculation will be left to the next subsection. 
\begin{theorem}
\label{thm:main thm for Radial}
Given growing rate $G_0>0$, background concentration $c_B>0$, nutrient consumption rate $\lambda>0$, and perturbation wave number $l\in\mathbb{N}^+$. When the radius of the tumor is around $R$ (the corresponding unperturbed tumor has radius $R$), the evolution function $\delta^{-1}\frac{d\delta}{d t}$ for the \emph{in vitro} model is given by:
\begin{equation}
\label{eqn:Radial invitro evolution fcn}
\delta^{-1}\frac{d\delta}{dt}=\frac{G_0 c_B I_1(\sqrt{\lambda}R)}{I_0(\sqrt{\lambda}R)}\left(
    \frac{I_1'(\sqrt{\lambda}R)}{I_1(\sqrt{\lambda}R)}-\frac{I_l'(\sqrt{\lambda}R)}{I_l(\sqrt{\lambda}R)}\right)+O(\delta)\defeq F_3(\lambda,l,R)+O(\delta).
\end{equation}
For the \emph{in vivo} model, $\delta^{-1}\frac{d\delta}{d t}$ is given by:
\begin{align}
\label{eqn:Radial invivo growing rate of delta}
\delta^{-1}\frac{d\delta}{dt}
&=\frac{G_0 c_B l}{\sqrt{\lambda}R C(R)}\left(\frac{C_1(R)}{C_l(R)}K_l(R)I_l(\sqrt{\lambda}R)-K_1(R)I_1(\sqrt{\lambda}R)\right)\\\nonumber
&\quad-\frac{G_0 c_B}{C(R)}\left(\frac{C_1(R)}{C_l(R)}K_l(R)I_l'(\sqrt{\lambda}R)-K_1(R)I_1'(\sqrt{\lambda}R)\right)+O(\delta)\\\nonumber
&\defeq F_4(\lambda,l,R)+O(\delta)\nonumber
\end{align}
where $C(R)$ is defined in \eqref{eqn:Radial a0 and b0}, and $C_j(R)$ ($j\in\mathbb{N}^+$) is given by
\begin{equation}
\label{eqn:Cj}
C_j(R)=K_j'(R)I_j(\sqrt{\lambda}R)-\sqrt{\lambda}I_j'(\sqrt{\lambda}R)K_j(R).
\end{equation}
\end{theorem}
Since the results are presented in terms of the Bessel functions, we review the basic properties of them in Appendix \ref{sec:Properties of Bessel functions}.
Also note that in either nutrient regime, the value of $G_0,c_B>0$ serve as scaling parameters, therefore do not influence the quantitative behavior of $\delta^{-1}\frac{d\delta}{dt}$. For the \emph{in vitro} model, we will show that $F_3(\lambda,l,R)$ is always negative. For the \emph{in vivo} model, fix the value of $G_0,c_B>0$, $F_4(\lambda,l,R)$ is plotted in Figure \ref{fig:Radial invivo} for different choice of $\lambda$ and perturbation wave number $l$.
\begin{figure}
  \begin{center}
    \includegraphics[width=0.48\textwidth]{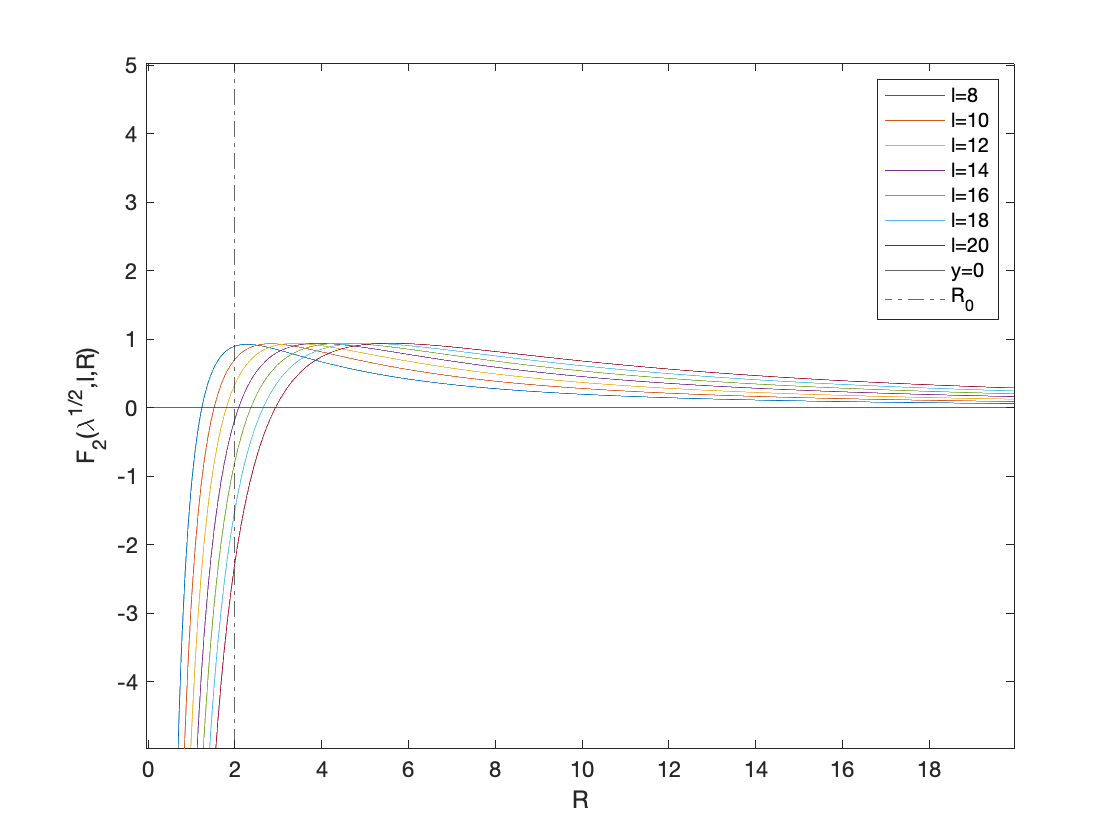}
    \includegraphics[width=0.48\textwidth]{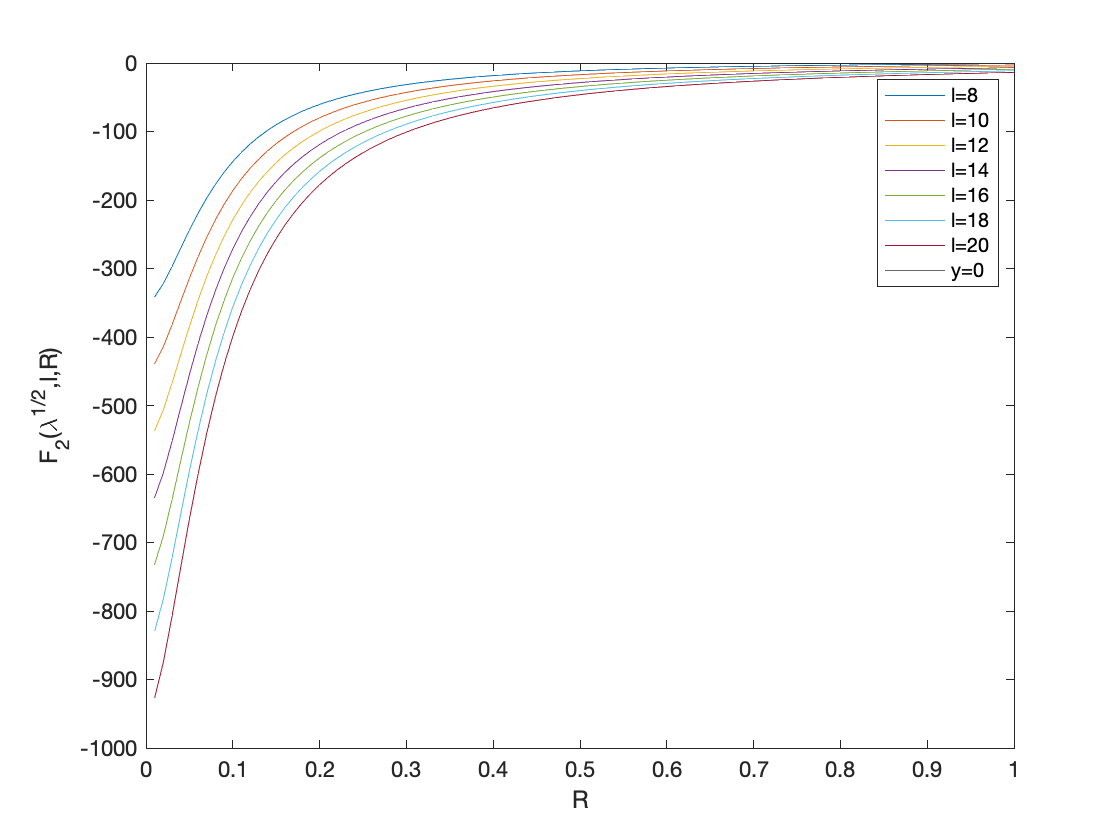}
     \includegraphics[width=0.48\textwidth]{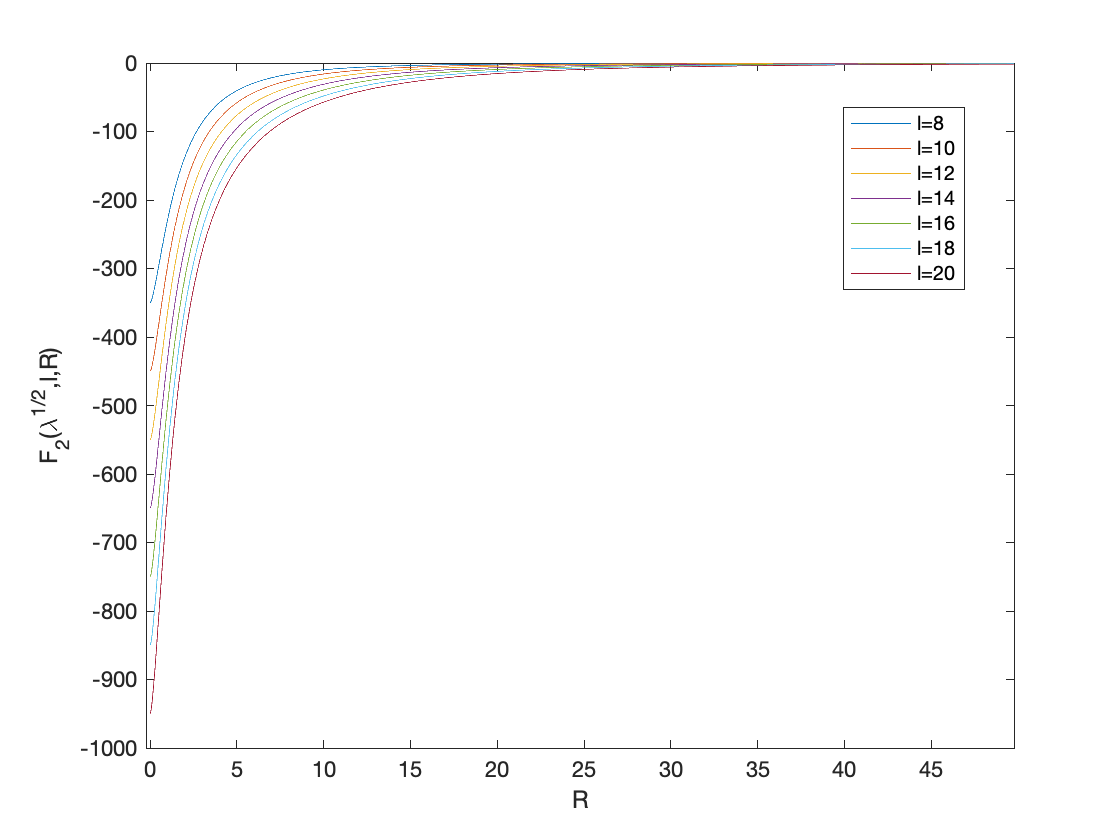}
     \includegraphics[width=0.48\textwidth]{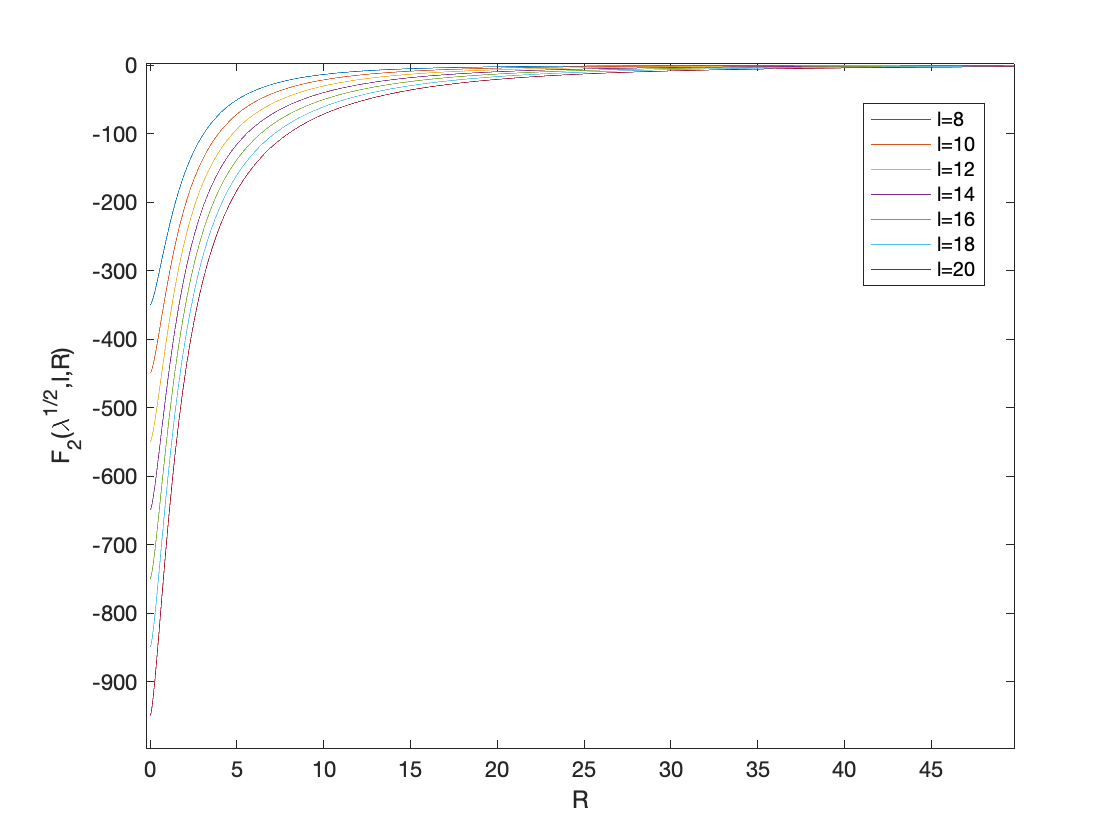}
     \caption{Graphs of $F_4$ with $G_0=1$, $c_B=100$; top (left): $\lambda=100$ and $R\in[0,20]$; top (right): $\lambda=100$ and $R\in[0,1]$; bottom (left): $\lambda=1$ and $R\in[0,50]$; bottom (right): $\lambda=0.8$ and $R\in[0,50]$.}
    \label{fig:Radial invivo}
  \end{center}
\end{figure}
Base on the expression of $\delta^{-1}\frac{d\delta}{dt}$ for the two nutrient models and the Figure \ref{fig:Radial invivo}, we establish following remarks.
\begin{remark}
$F_3(\lambda,1,R)=F_4(\lambda,1,R)=0$ for any $\lambda,R>0$. Since the mode $1$ perturbation corresponds to a trivial translation instead of the change of boundary geometry.
\end{remark}
\begin{remark}
\label{rmk:5.2}
When $0<\lambda\leq 1$, fix any wave number $l\geq 2$, $F_4(\lambda,l,R)$ is always negative and monotone increases in $R$. Physically, when the nutrient consumption rate is relatively low, the perturbation amplitude continuously decreases to zero, regardless of the perturbation wave number and tumor size. Correspondingly, the tumor always evolves from a star shape to a larger disk.
\end{remark}
\begin{remark}
\label{rmk:5.3}
For the regime $\lambda>1$, we have:
\begin{enumerate}
\item  For any fixed $l\geq 2$, there exists a threshold $R^*(l)$ such that $F_4(\lambda,l,R)<0$ for $0<R<R^*(l)$, and $F_4(\lambda,l,R)>0$ for $R>R^*(l)$ (see the left top picture in Figure \ref{fig:Radial invivo}). That means considering any single wave perturbation, and assume the nutrient consumption rate is significant, the perturbation amplitude will degenerate while the tumor size is relatively small, and the tumor will evolve from a star shape to a larger disk as in the \emph{in vitro} case. However, when the tumor size becomes large enough, the amplitude of the perturbation start to increase, and the tumor therefore remains in a star shape (but with a larger size). 
\item Fix a proper value of $R_0$, there exists $l_0$ such that $F_4(\lambda,l,R_0)>0$ for $l<l_0$ and $F_4(\lambda,l,R_0)<0$ for $l>l_0$ (see the left top picture in Figure \ref{fig:Radial invivo}), which implies that when the tumor size is around $R_0$ the perturbation of lower frequencies is easier to become unstable. See Figure \ref{fig:Radial invivo simulation} for the evolution of tumors under different perturbation wave numbers, where the blue curves correspond to the initial perturbed boundaries, and the red curves present the tumors evolve after a certain time.
\item As the tumor size expands, $R(t)$ exceeds more thresholds $R^*(l)$, therefore the corresponding wave number perturbation become unstable successively. 
\end{enumerate}
\end{remark}

\begin{figure}
  \begin{center}
    \includegraphics[width=0.48\textwidth]{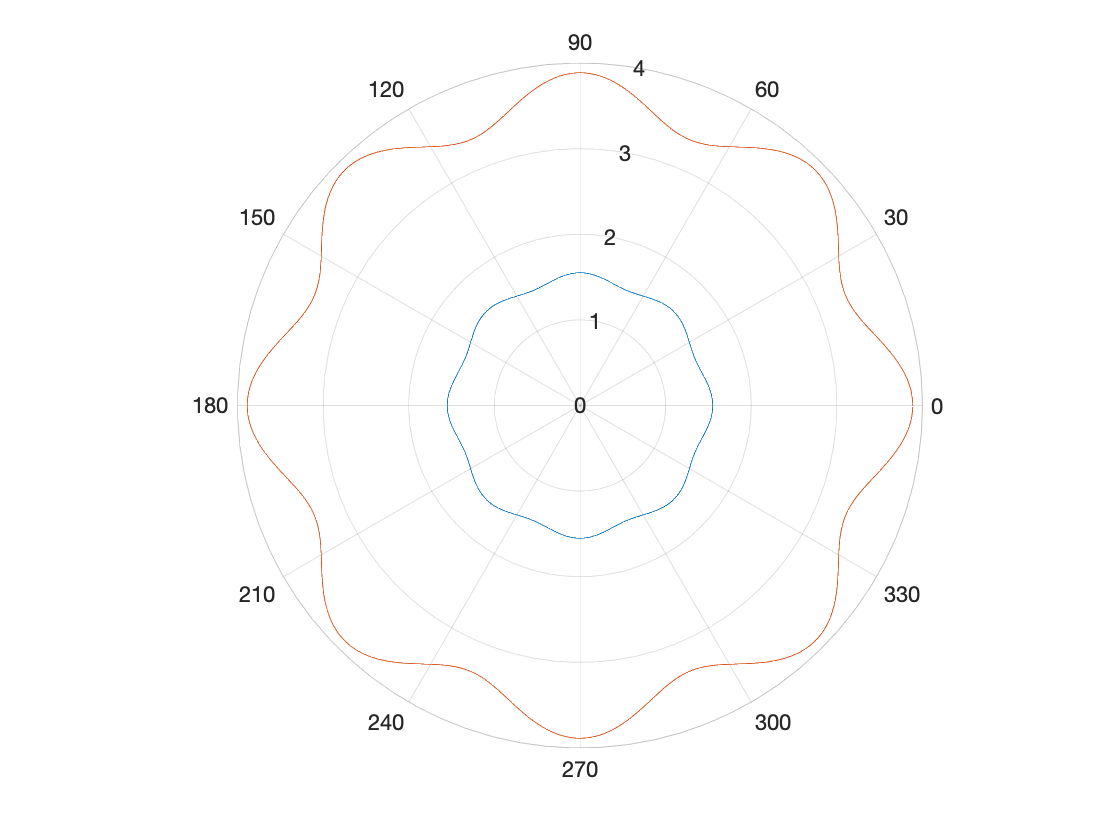}
    \includegraphics[width=0.48\textwidth]{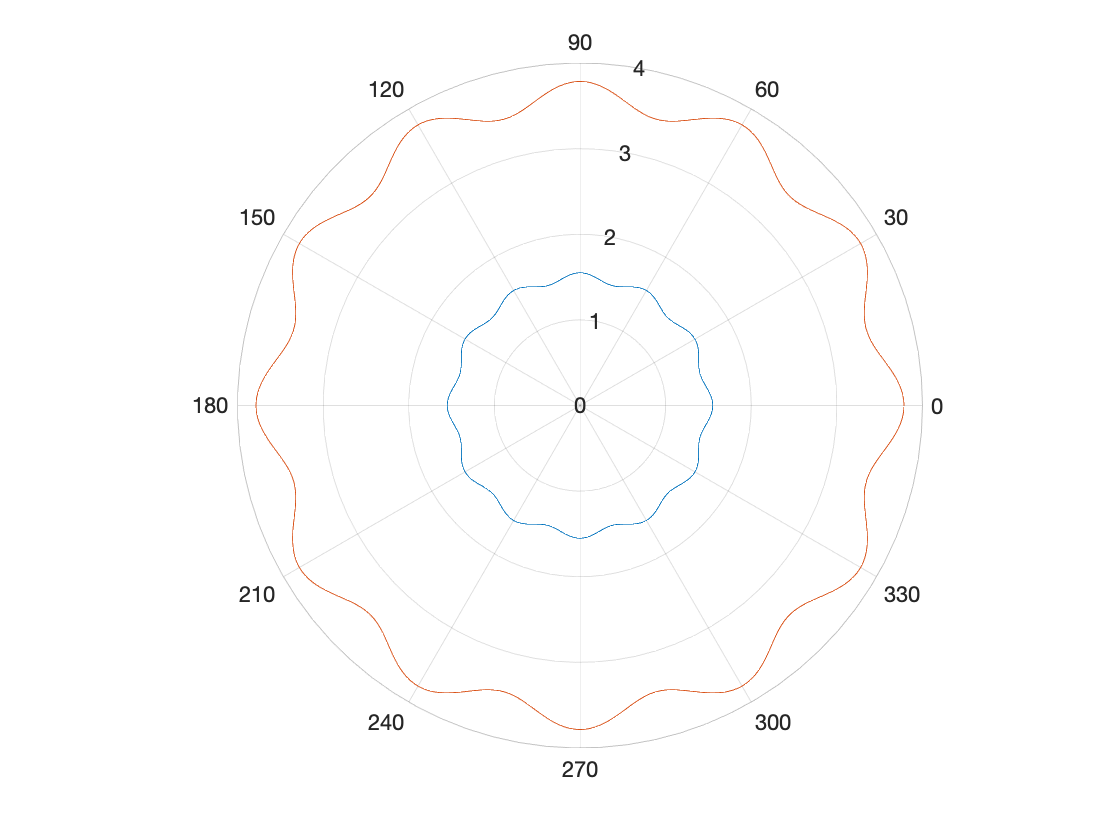}
    \includegraphics[width=0.48\textwidth]{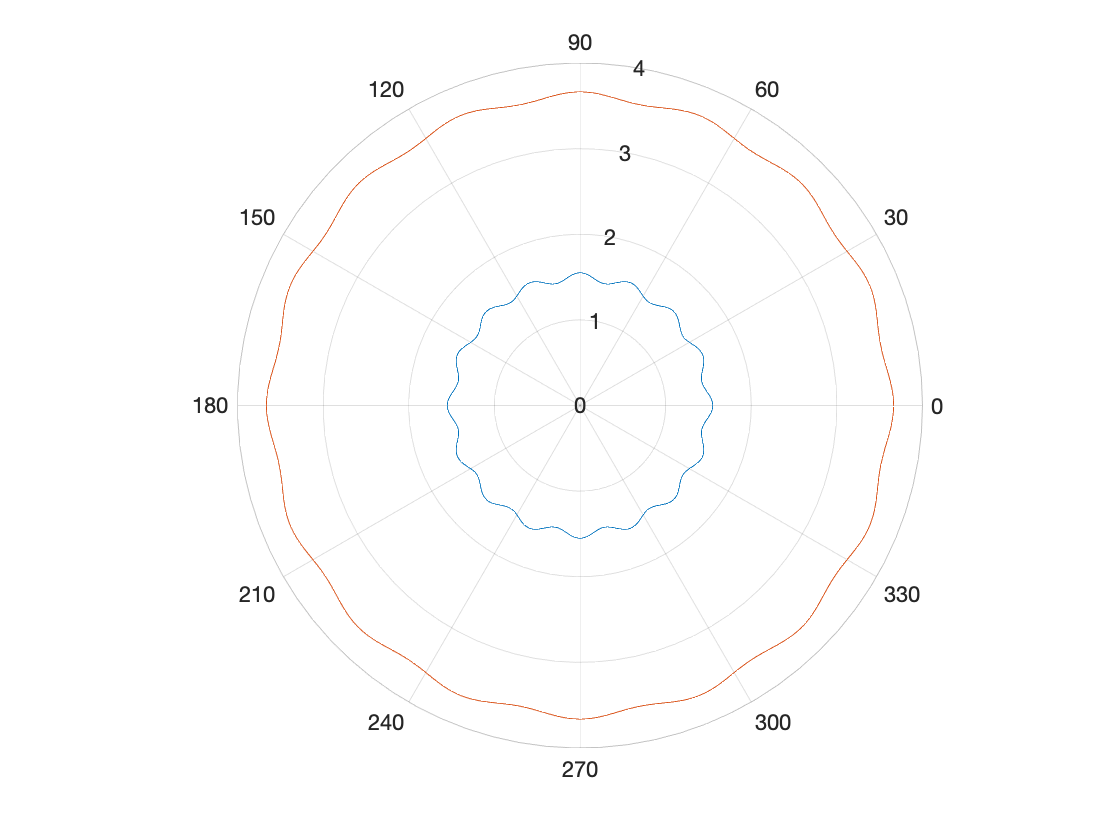}
    \includegraphics[width=0.48\textwidth]{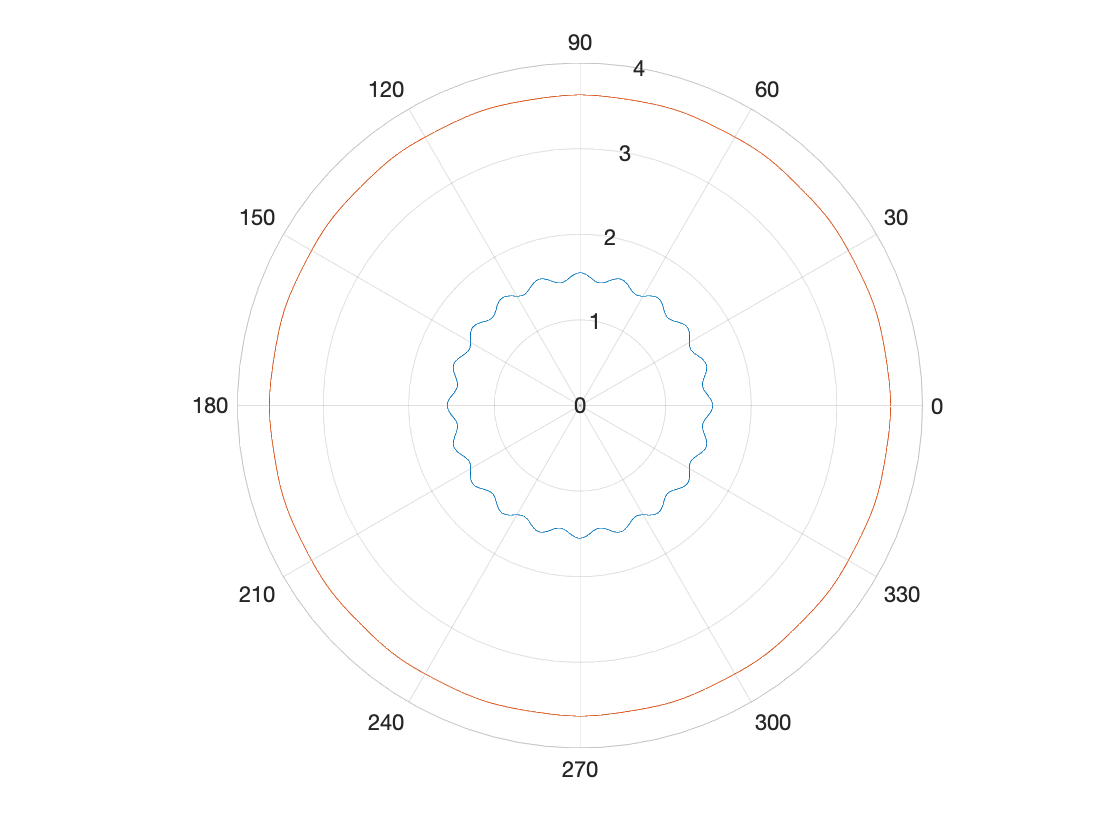}
    \caption{Evolution of tumor boundary for the \emph{in vivo} model with parameters: $G_0=1$, $c_B=100$, $\lambda=100$, $R_0=1.5$, $\delta=0.05$, $T=2$. Wave numbers from left to right and top to bottom: $l=8$, $l=12$, $l=16$, $l=20$.}
    \label{fig:Radial invivo simulation}
  \end{center}
\end{figure}
Some of the results in the above remarks can be proved rigorously, we summarize them in the following Corollary and the proof is left to Section \ref{sec:Radial Boundary stability analysis for the two nutrient models}.
\begin{corollary}
\label{cor:Radial}
Fix $G_0>0$ and $c_B>0$. For any $\lambda>0$ and $l>0$, $F_3(\lambda,l,R)$ is always negative, therefore the perturbation amplitude always decays for the \emph{in vitro} model. 
For the \emph{in vivo} model, we are able to show that for any $\lambda>0$, 
\begin{subequations}
\begin{alignat}{2}
F_4(\lambda,l,R)&\sim G_0 c_B\frac{1-l}{2}\qquad&\text{as} \quad R\sim 0,\\
F_4(\lambda,l,R)&\sim G_0 c_B\frac{5(l^2-1)(\sqrt{\lambda}-1)}{16\lambda R^2(\sqrt{\lambda+1})}\qquad&\text{as} \quad R\sim +\infty. 
\end{alignat}
\end{subequations}
Therefore, when $\lambda>1$ and $l\geq 2$ there exists $R^*(l)$ such that $F_4(\lambda,l,R^*(l))=0$. In addition, for $\lambda=1$ and $l\geq 2$, $F_4(1,l,R)$ can be simplified into a simpler form:
\begin{equation}
    \frac{1}{c_B G_0}F_4(1,l,R)
   =\sum_{j=1}^{l-1}R\Big(K_{j}(R)I_{j+1}(R)-K_{j+1}(R)I_{j+2}(R)\Big)-K_1(R)I_1(R).
\end{equation}
\end{corollary}
Note that to fully prove the observations in Remark \ref{rmk:5.2} and Remark \ref{rmk:5.3}, besides the asymptotes given in Corollary \ref{cor:Radial}, one also need to prove some monotonicity results of $F_4(\lambda,l,R)$ with respect to the variable $R$ or $l$. Unfortunately, we fail to carry out the proof of that even for the special case $\lambda=1$.

\subsection{The detailed calculations for the two nutrient regimes}
\label{sec:Radial The detailed calculations for the two nutrient regimes}
In this subsection we carry out the details of finding the expression of $\delta^{-1}\frac{d \delta}{dt}$ for the two nutrient models, which completes the proof of theorem \ref{thm:main thm for Radial}.

For the \emph{in vivo} case. Plugging the expansion \eqref{eqn:Radial expansion of c1} of $c_1(\xi,y)$ into \eqref{eqn:Radial invivo c1 eqns}, together with the conditions \eqref{eqn:Radial c_1 -infty} and \eqref{eqn:Radial c_1 +infty}, for any $k\in\mathbb{N}^+$ we have:
\begin{subequations}
\begin{alignat}{2}
    c^{\text{(i)},k}_1(r,t)&=c_B a_1^k(t) I_{kl}(\sqrt{\lambda}r)\qquad&\text{for}\quad r\leq R(t),\\
    c^{\text{(o)},k}_1(r,t)&=c_B b_1^k(t) K_{kl}(r)\qquad&\text{for}\quad r\geq R(t).
\end{alignat}
\end{subequations}
Recall that the leading order terms $c_0^{\text{(i)}}(\xi)$ and $c_0^{\text{(o)}}(\xi)$ are given by \eqref{eqn:Radial nutrient steady soln invivo}. Then \eqref{eqn:Radial invivo cts at tumor boundary a}-\eqref{eqn:Radial invivo cts at tumor boundary c} yield $a_1^k=b_1^k=0$ for any $k\neq 1$, since $I_j$ and $K_j$ have the same sign, but $I_j'$ and $K_j'$ have the opposite sign, while for $k=1$, we get nontrivial solutions
\begin{subequations}
\begin{alignat}{2}
    a_1^1(t)&=\frac{\left(a_0(R(t))\lambda I_1'(\sqrt{\lambda}R(t))+b_0(R(t)) K_1'(R(t))\right)K_l(R(t))}{C_l(R(t))},\\
    b_1^1(t)&=\frac{\left(a_0(R(t))\lambda I_1'(\sqrt{\lambda}R(t))+b_0(R(t)) K_1'(R(t))\right)I_l(\sqrt{\lambda}R(t))}{C_l(R(t))}.
\end{alignat}
where $C_l(R(t))$ is defined by \eqref{eqn:Cj}, and $a_0(R(t))$ and $b_0(R(t))$ are given by \eqref{eqn:Radial a0 and b0}. Note that $a_1^1(t)$ and $b_1^1(t)$ are depend on $t$ via the tumor radius $R(t)$, therefore we write $a_1^1(R)$ instead of $a_1^1(t)$ in the following, and similarly for $b_1^1(R)$.
\end{subequations}

By now, $c_1^k(r,t)$ is determined for all $k$, therefore $c_1(r,\theta,t)$ is also determined. Then by solving equation \eqref{eqn:Radial p1 equation} with expansion \eqref{eqn:Radial expansion of p1} and the boundary conditions in \eqref{eqn:Radial p1 boundary conditions} (with $p_0$ given by \eqref{eqn:Radial pressure steady soln invivo}), we get $p_1^k(r,t)=0$ for any $k\neq 1$, and:
\begin{subequations}
\begin{equation}
      p_1^1(r,t)=G_0 c_B \left(B_l(R) r^l-\frac{1}{\lambda}a_1^1(R)I_l(\sqrt{\lambda}r)\right)
\end{equation}
with $B_l(R)$ given by:
\begin{equation}
B_l(R)=\frac{1}{R^l}\left(\frac{a_1^1(R) I_l(\sqrt{\lambda}R)}{\lambda}+\frac{a_0(R) I_1(\sqrt{\lambda}R)}{\sqrt{\lambda}}\right).  
\end{equation}
\end{subequations}
By using the expression of $p_0(r,t)$ and $p_1^1(r,t)$, \eqref{eqn:change rate of delta} yields that up to an error of $O(\delta)$:
\begin{align}
    \delta^{-1}\frac{d\delta}{dt}
    &=-\left(\partial^2_{r}p_0(R,t)+\partial_{r}p_1^1(R,t)\right)\\\nonumber
    &=G_0 c_B\left(a_0(R) I_1'(\sqrt{\lambda}R)-B_{l}(R)l R^{l-1}+\frac{a_1^1(R) I_l'(\sqrt{\lambda}R)}{\sqrt{\lambda}}\right)\\\nonumber
    &=G_0 c_B \frac{l}{\sqrt{\lambda}R\cdot C(R)}\left(\frac{C_1(R)}{C_l(R)}K_l(R)I_l(\sqrt{\lambda}R)-K_1(R)I_1(\sqrt{\lambda}R)\right)\\\nonumber
    &\quad-G_0 c_B\frac{1}{C(R)}\left(\frac{C_1(R)}{C_l(R)}K_l(R)I_l'(\sqrt{\lambda}R)-K_1(R)I_1'(\sqrt{\lambda}R)\right)\\
    &\defeq F_4(\lambda,l,R).\nonumber
\end{align}
For the \emph{in vitro} model. Plugging the expansion of $c_1(r,\theta,t)$ \eqref{eqn:Radial expansion of c1} into \eqref{eqn:Radial c1 eqns invitro}, together with the conditions \eqref{eqn:Radial c_1 -infty}, for any $k\in\mathbb{N}^+$ we have:
\begin{equation}
    c^{k}_1(r,t)=c_B a_1^k(t) I_{kl}(\sqrt{\lambda}r)\qquad\text{for}\quad r\leq R(t).
\end{equation}
And the leading order term $c_0(r,t)$ for this case is given by \eqref{eqn:Radial nutrient steady soln invitro}. Then by using boundary condition \eqref{eqn:Radial invitro boundary condtion}, we get $a^k_1(t)=0$ for any $k\neq 1$, since $I_j$ is always positive. While for $k=1$:
\begin{equation}
    a^1_1(R(t))=-\frac{c_B \sqrt{\lambda} I_1(\sqrt{\lambda}R)}{I_0(\sqrt{\lambda}R) I_l(\sqrt{\lambda}R)}.
\end{equation}
Then similar to the previous case, by solving equation \eqref{eqn:Radial p1 equation} together with boundary conditions \eqref{eqn:Radial p1 boundary conditions} (with $p_0$ given by \eqref{eqn:Radial pressure steady soln invitro}), we get $p_1^k(r,t)=0$ for any $k\neq 1$. And for $k=1$:
\begin{equation}
     p^1_1(r,t)=\frac{G_0 c_B I_1(\sqrt{\lambda}R)}{\sqrt{\lambda}I_0(\sqrt{\lambda}R) I_l(\sqrt{\lambda}R)}I_l(\sqrt{\lambda}r).
\end{equation}
Finally, by using the expression of $p_0$ and $p_1^1$, \eqref{eqn:change rate of delta} yields up to an error of $O(\delta)$:
\begin{align}
\label{eqn:Radial invitro growing rate of delta}
    \delta^{-1}\frac{d\delta}{dt}
    &= -\left(\partial^2_{r}p_0(R,t)+\partial_{r}p_1^1(R,t)\right)\\\nonumber
    &=\frac{G_0 c_B I_1(\sqrt{\lambda}R)}{I_0(\sqrt{\lambda}R)}\left(
    \frac{I_1'(\sqrt{\lambda}R)}{I_1(\sqrt{\lambda}R)}-\frac{I_l'(\sqrt{\lambda}R)}{I_l(\sqrt{\lambda}R)}\right)=F_3(\lambda,l,R).
\end{align}
\subsection{Boundary stability analysis for the two nutrient models}
\label{sec:Radial Boundary stability analysis for the two nutrient models}
In this subsection, we prove the mathematical properties of $F_3$ and $F_4$ summarized in Corollary \ref{cor:Radial} by using the properties for Bessel functions in the Appendix \ref{sec:Properties of Bessel functions}.

For the \emph{in vitro} model, $\delta^{-1}\frac{d\delta}{dt}$ is given by \eqref{eqn:Radial invitro evolution fcn}, which is negative for any $l\in\mathbb{N}^+$ and $c_B,G_0,\lambda,R>0$. Indeed, observe that \eqref{eqn:Radial invitro evolution fcn} can be written as:
\begin{equation}
    \delta^{-1}\frac{d\delta}{dt}=G_0 c_B\frac{I_1(\sqrt{\lambda}R)}{\sqrt{\lambda}R I_0(\sqrt{\lambda}R)}H_l(\sqrt{\lambda}R),
\end{equation}
where $H_l(r)\defeq r\left(\frac{I_1'(r)}{I_1(r)}-\frac{I_l'(r)}{I_l(r)}\right)$. It was checked that $H'_l(r)>0$ for any $r>0$ and $l\in\mathbb{N}+$ (see equation (2.19) in \cite{friedman2001symmetry}). On the other hand, by using the asymptote of $I_l(r)$ in \eqref{eqn:asymptote of I and K infty}, one can check that $\lim_{r\rightarrow\infty}H_l(r)=0$. Thus, $H_l(r)<0$ for any $r>0$ and wave number $l$, which further yields $\delta^{-1}\frac{d\delta}{d t}$ is negative as well. Therefore, the amplitude of the perturbation decays as time evolves for any wave number. By now, the proof of the argument for the \emph{in vitro} model in Corollary \ref{cor:Radial} is completed.

For the \emph{in vivo} model, $\delta^{-1}\frac{d\delta}{d t}$ is given by \eqref{eqn:Radial invivo growing rate of delta}. Now we check the properties of $F_4(\lambda,l,R)$ established in Corollary \ref{cor:Radial}. Observe that according to \eqref{eqn:Radial invivo growing rate of delta}, the evolution function can be decomposed as $F_4(\lambda,l,R)=G_0 c_B (T_1-T_2)$, where $T_1(\lambda,l,R)$ and $T_2(\lambda,l,R)$ are given by:
\begin{subequations}
\begin{alignat}{2}
T_1(\lambda,l,R)&=\frac{l}{\sqrt{\lambda}R C(R)}\left(\frac{C_1(R)}{C_l(R)}K_l(R)I_l(\sqrt{\lambda}R)-K_1(R)I_1(\sqrt{\lambda}R)\right),\\
T_2(\lambda,l,R)&=\frac{1}{C(R)}\left(\frac{C_1(R)}{C_l(R)}K_l(R)I_l'(\sqrt{\lambda}R)-K_1(R)I_1'(\sqrt{\lambda}R)\right).
\end{alignat}
\end{subequations}
By using the asymptotes in \eqref{eqn:asymptote of I and K 0} one can check that for any wave number $l\geq 2$:
\begin{subequations}
\begin{alignat}{2}
&T_1\sim\frac{1-l}{2},&\qquad\text{as}\quad R\sim 0,\\
&T_2\sim 0,&\qquad\text{as}\quad R\sim 0.
\end{alignat}
Therefore, $F_4(\lambda,l,R)\sim G_0 c_B\frac{1-l}{2}<0$ as $R$ approaches to zero (this can be observed in Figure \ref{fig:Radial invivo}). 

On the other hand, by using the asymptotes at infinity: \eqref{eqn:asymptote of I and K infty} and \eqref{eqn:asymptote of I' and K' infty}, we can also check that
\begin{subequations}
\begin{alignat}{2}
&T_1\sim\frac{l(1-10 l^2)}{32\lambda R^3(\sqrt{\lambda}+1)}=O(\frac{1}{R^3}),&\qquad\text{as}\quad R\sim+\infty,\\
&T_2\sim\frac{5(1-l^2)(\sqrt{\lambda}-1)}{16\lambda R^2(\sqrt{\lambda+1})},&\qquad\text{as}\quad R\sim +\infty,
\end{alignat}
\end{subequations}
which further yields
\begin{equation}
 F_4(\lambda, l, R)\sim \frac{5(l^2-1)(\sqrt{\lambda}-1)}{16\lambda R^2(\sqrt{\lambda+1})}+O(\frac{1}{R^3})\quad\text{as}\quad R\sim +\infty.
\end{equation}
\end{subequations}
Therefore, for $\lambda>1$ by intermediate value theorem the function $F_4(\sqrt{\lambda},l,R_0)$ must intersect the horizontal axis.

Now, we investigate the special case $\lambda=1$. Firstly, observe that by using the identity
\eqref{eqn:bessel identity}, $F_4(\lambda,l,R)$ can be further written into:
\begin{align}
    F_4=\frac{1}{C(R)}\left(K_1(R)I_2(\sqrt{\lambda}R)-\frac{C_1(R)}{C_l(R)}K_l(R)I_{l+1}(\sqrt{\lambda}R)-\frac{l-1}{\sqrt{\lambda}R}K_1(R)I_1(\sqrt{\lambda}R)\right).
\end{align}
Also note that when $\lambda=1$, by using the Wronskians cross product \eqref{eqn:Wronskians}, we get:
\begin{subequations}
\begin{alignat}{2}
C(R)&=I_0(R)K_{1}(R)+I_{1}(R)K_0(R)=\frac{1}{R},\\
C_l(R)&=-\frac{1}{2}\left(\left(K_{l+1}(R)+K_{l-1}(R)\right)I_l(R)+\left(I_{l+1}(R)+I_{l-1}(R)\right)K_l(R)\right)\\\nonumber
&=-\frac{1}{R}\quad\text{for}\quad\forall l\in\mathbb{N}.
\end{alignat}
\end{subequations}
Therefore, when $\lambda=1$ we can further simplify $F_4(\lambda,l,R)$ into:
\begin{align}
  \frac{1}{c_B G_0}F_4(\lambda,l,R)
  &=R\Big(K_1(R)I_2(R)-K_{l}(R)I_{l+1}(R)\Big)-(l-1)K_1(R)I_1(R)\\\nonumber
  &=\sum_{j=1}^{l-1}R\Big(K_{j}(R)I_{j+1}(R)-K_{j+1}(R)I_{j+2}(R)\Big)-K_1(R)I_1(R).
\end{align}
\subsection{Relationship between the Radial boundary and the traveling wave boundary}
\label{sec:Relationship between the Radial boundary and the traveling wave boundary}
In the last section we discuss the relationship between the two kinds of boundaries. In Section \ref{sec:Analytic solutions} we have already checked that without the perturbation, the propagation speed of the radial boundary converges to that of the traveling wave boundary as the radius tends to infinity. 

Now, we explore the relationship for the perturbed boundaries. As before, we use $(r,\theta)$ to present the polar coordinates, and $(\xi,y)$ for the Euler coordinates. Considering the perturbation of the radial boundary, let $\Tilde{y}=\theta*R(t)$ and $\Tilde{l}=l/R(t)$. Then the perturbation part can be rewritten as:
\begin{equation}
    \mathcal{P}_l(\theta)=\cos{l\theta}=\cos{\Tilde{l}\Tilde{y}}\defeq\mathcal{P}_{\omega}(\Tilde{y})
\end{equation}
with $\Tilde{y}\in (-\pi R,\pi R)$. Moreover, as $R\rightarrow +\infty$, $\Tilde{l}$ tends to zero and $\Tilde{y}\in\mathbb{R}$.

Also note that we can map the unperturbed radial boundary, $r=R$, to the unperturbed traveling wave boundary, $\xi=0$, by the map:
\begin{equation}
  (R,\theta)\mapsto(0, \tan{\theta/2}),\qquad \theta\in (-\pi,\pi).
\end{equation}
Thus, as $R\rightarrow+\infty$ any radial perturbation with finite wave number $l$ (defined by \eqref{eqn:Radial boundary perturbation}) will converge to the perturbation of the traveling wave boundary (defined by \eqref{eqn:TW boundary perturbation}) but with the zero frequency. Further more, for the same nutrient model the following relationships of the amplitude evolution equations hold:
\begin{subequations}
\begin{alignat}{2}
\lim_{R\rightarrow+\infty}F_3(\lambda,l,R)&=F_1(\lambda,0)=0,\\
\lim_{R\rightarrow+\infty}F_4(\lambda,l,R)&=F_2(\lambda,0)=0.
\end{alignat}
\end{subequations}
for any $\lambda>0$ and $l\in\mathbb{N}^+$.
\section{Conclusion}
\label{Conclusion}
In this paper, we study the tumor boundary instability induced by nutrient consumption and supply in two scenarios: 1) the front of the traveling wave; 2) the radially symmetric boundary. In each scenario, we investigate the boundary behaviors under two different nutrient supply regimes, \emph{in vitro} and \emph{in vivo}. 

For the traveling wave scenario, our analysis shows the boundary is stable for any frequency perturbation $l\in\mathbb{R}^+$ and positive consumption rate $\lambda$ when the nutrient supply is governed by the \emph{in vitro} regime. In contrast, for \emph{in vivo} regime, there exists a threshold value $L$ such that the perturbation with a frequency smaller than $L$ becomes unstable when the nutrient consumption rate $\lambda$ is larger than one. 

Then we consider the radially symmetric boundary scenario to explore further the influence of the finite size effect on boundary stability/instability. Our analysis shows that the \emph{in vitro} regime still suppresses the increase of perturbation amplitude and stabilizes the boundary regardless of the consumption rate $\lambda$, perturbation wave number $l\in\mathbb{N}$, and tumor size $R$. For the \emph{in vivo} regime, when $\lambda\leq 1$, the boundary behaves identically the same as the \emph{in vitro} case. However, when $\lambda>1$, the continuous growth of tumor radius enables perturbation wave number $l$ to become unstable in turn (from low to high). Further more, as $R$ is approaching infinity, the results in the radial case connect to the counterparts in the traveling wave case.

In the end, we conjecture that symmetric breaking traveling wave solutions may exist in the \emph{in vivo} nutrient regime. From Figure \ref{fig:F2}, one can observe that for proper large $l$ there exists $\lambda_0>1$ such that $F_2(\lambda_0,l)=0$, i.e., the perturbation amplitude $\delta$, up to some higher order error, neither growing nor decaying for such parameters. Thus, it is reasonable to expect that one may get the symmetric breaking traveling wave solutions by carefully modifying the linear solutions around the parameter $(\lambda_0,l)$. 
We speculate that the technique in \cite{friedman2001symmetry} might be helpful in solving this conjecture, which we save for future studies.

\section*{Acknowledgments}
The work of Y.F. is supported by the National Key R\&D Program of China, Project Number 2021YFA1001200. The work of M.T. is partially supported by Shanghai Pilot Innovation project, Project Number 21JC1403500, and NSFC grant number 11871340. The work of X.X. is supported by the National Key R\&D Program of China, Project Number 2021YFA1001200, and the NSFC Youth program, grant number 12101278. The work of Z.Z. is supported by the National Key R\&D Program
of China, Project Number 2021YFA1001200, 2020YFA0712000, and NSFC grant number 12031013, 12171013.
\newpage
\appendix
\section{Properties of Bessel functions}
\label{sec:Properties of Bessel functions}
Since the solutions of the radial case are presented in terms of the second kind modified Bessel functions $I_n(r)$ and $K_n(r)$ (for $n\in\mathbb{N}$), we review some basic properties of them in this section. Firstly, $I_n(r)$ and $K_n(r)$ solve the differential equation:
\begin{equation}
    r^2\frac{d^2f}{d r^2}+r\frac{d f}{d r}-(r^2+n^2)f=0,
\end{equation}
and are strict positive for any $n\in\mathbb{N}$ and $r>0$. For the derivatives, we have $I_0'(r)=I_1(r)>0$ and $K_0'(r)=-K_1(r)<0$, and for $n\geq 1$:
\begin{subequations}
\label{eqn:I' and K'}
\begin{alignat}{2}
I_n'(r)&=\frac{I_{n-1}(r)+I_{n+1}(r)}{2}>0,\\
K_n'(r)&=-\frac{K_{n-1}(r)+K_{n+1}(r)}{2}<0.
\end{alignat}
Therefore, $I_j(r)$ are monotone increasing functions, and $K_j(r)$ are monotone decreasing functions.
\end{subequations}
And the Bessel function $I_n(r)$ satisfies:
\begin{equation}
\label{eqn:bessel identity}
 I_n'(r)-\frac{n}{r}I_n(r)=I_{n+1}(r),   
\end{equation}
for any $n\in\mathbb{N}^+$.

When $r\rightarrow 0$, $I_n(r)$ and $K_n(r)$ possess following asymptotes: 
\begin{subequations}
\label{eqn:asymptote of I and K 0}
\begin{alignat}{2}
&I_n(r)\sim\frac{1}{\Gamma(n+1)}(\frac{r}{2})^n,\qquad\text{for}\quad n\in\mathbb{N},\\
&K_n(r)\sim\frac{\Gamma(n)}{2}(\frac{r}{2})^{-n},\qquad\text{for}\quad n\in\mathbb{N}^+,\\
&K_0(r)\sim-\ln{r}.
\end{alignat}
\end{subequations}
While as $r\rightarrow +\infty$, $I_n(r)$ and $K_n(r)$ have the asymptotes: 
\begin{subequations}
\label{eqn:asymptote of I and K infty}
\begin{alignat}{2}
&I_n(r)\sim(\frac{1}{2\pi r})^{1/2}e^{r}\left(1-\frac{4n^2-1}{8r}+\frac{(4n^2-1)(4n^2-9)}{128 n^2}+O(\frac{1}{r^3})\right),\\
&K_n(r)\sim(\frac{\pi}{2r})^{1/2}e^{-r}\left(1+\frac{4n^2-1}{8r}+\frac{(4n^2-1)(4n^2-9)}{128 r^2}+O(\frac{1}{r^3})\right).
\end{alignat}
\end{subequations}
By using \eqref{eqn:I' and K'}, we can also derive the asymptotes for $I'(r)$ and $K'(r)$ for $r\rightarrow+\infty$:
\begin{subequations}
\label{eqn:asymptote of I' and K' infty}
\begin{alignat}{2}
&I_n'(r)\sim(\frac{1}{2\pi r})^{1/2}e^{r}\left(1-\frac{4n^2+3}{8r}+\frac{(4n^2-1)(4n^2+15)}{128 r^2}+O(\frac{1}{r^3})\right),\\
&K_n'(r)\sim-(\frac{1}{2\pi r})^{1/2}e^{-r}\left(1+\frac{4n^2+3}{8r}+\frac{(4n^2-1)(4n^2+15)}{128 r^2}+O(\frac{1}{r^3})\right).
\end{alignat}
\end{subequations}
Further more, $I_n(r)$ and $K_n(r)$ satisfy the so-called Wronskians cross product:
\begin{equation}
\label{eqn:Wronskians}
    I_n(r)K_{n+1}(r)+I_{n+1}(r)K_n(r)=\frac{1}{r}\quad\text{for}\quad\forall n\in\mathbb{N}.
\end{equation}

\newpage
\bibliographystyle{plain}
\bibliography{refs,preprints}

\end{document}